\documentclass[reqno]{amsart}
\usepackage{amsfonts,amsmath,amssymb}

\numberwithin{equation}{section}

\newtheorem{thm}{Theorem}[section]
\newtheorem*{thm*}{Theorem}
\newtheorem{lem}{Lemma}[section]

\newcommand{\fincal}{\end{eqnarray*}}

\begin{document}
\title[EKGMP systems in closed manifolds]
{Static Klein-Gordon-Maxwell-Proca systems 
in $4$-dimensional closed manifolds}
\author{Emmanuel Hebey}
\address{E. Hebey, Universit\'e de Cergy-Pontoise, 
D\'epartement de Math\'ematiques, Site de 
Saint-Martin, 2 avenue Adolphe Chauvin, 
95302 Cergy-Pontoise cedex, 
France}
\email{Emmanuel.Hebey@math.u-cergy.fr}
\author{Trong Tuong Truong}
\address{T.T. Truong, Universit\'e de Cergy-Pontoise, 
D\'epartement de Physique, Site de 
Saint-Martin, 2 avenue Adolphe Chauvin, 
95302 Cergy-Pontoise cedex, 
France}
\email{tuong.truong@u-cergy.fr}
\thanks{To appear in J. Reine Angew. Math.}

\begin{abstract} We prove existence and uniform bounds  
for critical static Klein-Gordon-Maxwell-Proca systems in the case of $4$-dimensional 
closed Riemannian manifolds.  
\end{abstract}

\maketitle

Static Klein-Gordon-Maxwell-Proca systems are massive versions of the electrostatic 
Klein-Gordon-Maxwell Systems. The vector field in these systems inherits a mass and 
is governed by the Proca action 
which generalizes that of Maxwell. 
Klein-Gordon-Maxwell systems are intended to provide a dualistic model for the description of 
the interaction between a charged relativistic matter scalar field and the electromagnetic field that it generates. The electromagnetic field is 
both generated by and drives the 
particle field. In the electrostatic form of the Klein-Gordon-Maxwell systems, looking for standing waves $ue^{i\omega t}$,  
the matter field is characterized by the property that 
$u$, together with a gauge potential $v$, solve the electrostatic Klein-Gordon-Maxwell systems \eqref{SWSyst} with $m_1 = 0$. In the case 
of a closed manifold we discuss here 
the two equations in \eqref{SWSyst} are independent one of another when $m_1 = 0$ and the system reduces to the sole Schr\"odinger equation. The 
Proca formalism, for $m_1 > 0$, leads to a deeper phenomenon and is more appropriate to the closed case. 
The particle in this model interacts via the minimum coupling rule
\begin{equation}\label{MinSubsRule}
\partial_t \to \partial_t+iq\varphi\hskip.2cm\hbox{and}\hskip.2cm \nabla \to \nabla - iqA
\end{equation}
with an external massive vector field $(\varphi,A)$ which is governed by the Maxwell-Proca Lagrangian. 
The Proca action is a gauge-fixed version 
of the Stueckelberg action in the Higgs mechanism (see Goldhaber and Nieto \cite{GolNie2}, and 
Ruegg and Ruiz-Altaba \cite{RueRui}).
In the Proca formalism, developped under the influence of de Broglie, the photon inherits a nonzero  mass. This issue is of considerable importance 
and intensively studied in modern physics (see for instance Adelberger, Dvali and Gruzinov \cite{AdeDvaGru}, Byrne \cite{Byr}, 
Goldhaber and Nieto \cite{GolNie1,GolNie2}, Luo and Tu \cite{LuoTu}, Luo, Gillies and Tu \cite{LuoGilTu} 
 and the references in these papers). When $n = 3$, the KGMP equations consist 
in the nonlinear Klein-Gordon matter equation, the charge continuity equation and the massive modified 
Maxwell equations in SI units, which are hereafter explicitly written down:
\begin{equation}\label{MaxwellProca}
\begin{split}
&\nabla.E = \rho/\varepsilon_0 - \mu^2\varphi\hskip.1cm ,\\
&\nabla\times H = \mu_0\left(J+\varepsilon_0\frac{\partial E}{\partial t}\right) - \mu^2A\hskip.1cm ,\\
&\nabla\times E + \frac{\partial H}{\partial t} = 0\hskip.1cm\hbox{and}\hskip.1cm \nabla. H = 0\hskip.1cm .
\end{split}
\end{equation}
These massive Maxwell equations, as modified to Proca form, appear to have been first written in modern format by Schr\"odinger \cite{Sch}.  
The Proca formalism a priori breaks Gauge invariance. Gauge invariance can be restaured by the Stueckelberg trick, as pointed out by Pauli \cite{Pau}, and then by 
the Higgs mechanism. We refer to 
Goldhaber and Nieto \cite{GolNie2}, Luo, Gillies and Tu \cite{LuoGilTu}, and Ruegg and Ruiz-Altaba \cite{RueRui} for very complete 
references on the Proca approach.

\medskip In what follows we let $(M,g)$ be a smooth compact $3, 4$-dimensional Riemannian manifold. We let also $2^\star = \frac{2n}{n-2}$ 
be the critical Sobolev exponent, where $n$ is the dimension of $M$. Given real numbers $q > 0$, 
$m_0, m_1 > 0$, $\omega \in (-m_0,m_0)$, and $p \in (2,2^\star]$, the derivation of the Klein-Gordon-Maxwell-Proca system we investigate 
in this paper is written as
\begin{equation}\label{SWSyst}
\begin{cases}
\Delta_gu + m_0^2u = u^{p-1} + \omega^2\left(qv-1\right)^2u\\
\Delta_gv + \left(m_1^2+q^2u^2\right)v = qu^2\hskip.1cm ,
\end{cases}
\end{equation}
where $\Delta_g = -\hbox{div}_g\nabla$ is the Laplace-Beltrami operator. The system \eqref{SWSyst} 
corresponds to looking for standing waves $ue^{i\omega t}$ for the full KGMP system in the static case where the 
massive vector field $(\varphi, A)$ depends on the sole spatial variable. The system 
is energy critical when $n = 3$ and $p = 6$ and when 
$n = 4$ and $p = 4$. It is subcritical otherwise, namely when $n = 3$ and $p \in (2,6)$ or $n=4$ and  
$p \in (2,4)$. In the above model, $m_1$ is a coupling constant which 
makes that the two equations in \eqref{SWSyst} are trully coupled ($m_1$ is the Proca mass in the Maxwell-Proca formalism) while 
$m_0$ is the mass of the particle, $q$ is the charge of the particle, $u$ is the amplitude in the writing of the particle, 
$\omega$ is its temporal frequency (referred to as the phase in the sequel), and $v$ is the electric potential. 

\medskip Let $S_g$ stand for the scalar curvature of $g$, and $\mathcal{S}_p(\omega)$ be 
the set consisting of the positive smooth solutions $\mathcal{U} = (u,v)$ 
of \eqref{SWSyst} with phase $\omega$ and nonlinear term $u^{p-1}$. Namely,
\begin{equation}\label{DefSp}
\mathcal{S}_p(\omega) = \Bigl\{(u,v)\hskip.1cm\hbox{smooth}\hskip.1cm\hbox{s.t.}\hskip.1cm u > 0, v > 0,\hskip.1cm\hbox{and}\hskip.1cm (u,v)
\hskip.1cm\hbox{solve}\hskip.1cm\eqref{SWSyst}\Bigr\}\hskip.1cm .
\end{equation}
Given $\omega \in [0,m_0)$, we let
\begin{equation}\label{DefKOm0}
K_0(\omega) = (-m_0,-\omega]\bigcup [\omega,m_0)\hskip.1cm .
\end{equation}
When $\omega = 0$, 
$K_0(0) = (-m_0,m_0)$ is the full admissible phase range. 
For $\theta \in (0,1)$, and $\mathcal{U} = (u,v)$, we let 
$\Vert\mathcal{U}\Vert_{C^{2,\theta}} = \Vert u\Vert_{C^{2,\theta}} + \Vert v\Vert_{C^{2,\theta}}$. 
By a MPT solution we mean a solution with a strong mountain pass type structure. 
The following result was proved in Druet and Hebey \cite{DruHeb}. 

\begin{thm}[The $3$-dimensional case - Druet and Hebey \cite{DruHeb}]\label{Thm0} Let $(M,g)$ be a smooth compact $3$-dimensional Riemannian manifold 
$m_0, m_1 > 0$, $\omega \in (-m_0,m_0)$, and $p \in (2,6]$. When $p = 6$ assume
\begin{equation}\label{MainAssumpt0}
m_0^2 < \omega^2 + \frac{1}{8}S_g(x)
\end{equation}
for all $x \in M$. Then \eqref{SWSyst} possesses a smooth positive MPT solution. 
Moreover, for any $p \in (2,6)$, and any $\theta \in (0,1)$, there exists $C >0$ such that for any $\omega^\prime \in K_0(0)$, and 
any $\mathcal{U} \in \mathcal{S}_p(\omega^\prime)$, $\Vert\mathcal{U}\Vert_{C^{2,\theta}} \le C$, where $\mathcal{S}_p(\omega^\prime)$ 
is as in \eqref{DefSp} and $K_0(0)$ is as in \eqref{DefKOm0}. 
Assuming again \eqref{MainAssumpt0}, there also holds that for any $\theta \in (0,1)$, 
$\Vert\mathcal{U}\Vert_{C^{2,\theta}} \le C$ for all $\mathcal{U} \in \mathcal{S}_6(\omega^\prime)$ 
and all $\omega^\prime \in K_0(\omega)$, where $C > 0$ does not depend on $\omega^\prime$ and $\mathcal{U}$.
\end{thm}

This result exhibits phase compensation in the $3$-dimensional case. We aim in this paper in proving that a similar phenomenon 
holds true when $n = 4$. In this dimension the second equation in \eqref{SWSyst} becomes critical and this leads to 
serious difficulties. We prove below the existence of smooth positive MPT solutions and 
the existence of uniform bounds for \eqref{SWSyst} in the subcritical cases $p \in (2,4)$ without any conditions, 
and in the critical case $p =4$ 
assuming that the mass potential, balanced by the phase, 
is smaller than the geometric threshold potential of the conformal Laplacian. In doing so we prove that phase compensation still holds true for our systems 
when $n = 4$. 
Our result, in the subcritical case, is as follows.

\begin{thm}[The subcritical $4$-dimensional case]\label{Thm1} Let $(M,g)$ be a smooth compact $4$-dimensional Riemannian manifold, $q > 0$, 
$m_0, m_1 > 0$, $\omega \in (-m_0,m_0)$, and $p \in (2,4)$. 
Then \eqref{SWSyst} possesses a smooth positive MPT solution. Moreover, 
for any $\theta \in (0,1)$, there exists $C >0$ such that for any $\omega^\prime \in K_0(0)$, and 
any $\mathcal{U} \in \mathcal{S}_p(\omega^\prime)$, $\Vert\mathcal{U}\Vert_{C^{2,\theta}} \le C$, where $\mathcal{S}_p(\omega^\prime)$ 
is as in \eqref{DefSp} and $K_0(0)$ is as in \eqref{DefKOm0}.
\end{thm}

In the critical case we prove the following result. The geometry of the ambiant inhomogeneous space, through the scalar 
curvature of $g$, comes to play a role as in the $3$-dimensional case. However, the result now turns out to be local in its existence 
part. 

\begin{thm}[The critical $4$-dimensional case]\label{Thm2} Let $(M,g)$ be a smooth compact $4$-dimensional Riemannian manifold, $q > 0$, 
$m_0, m_1 > 0$, $\omega \in (-m_0,m_0)$, and $p=4$. Assume
\begin{equation}\label{MainAssumpt}
m_0^2 < \omega^2 + \frac{1}{6}S_g(x)
\end{equation}
for some $x \in M$. Then \eqref{SWSyst} possesses a smooth positive MPT solution. 
Assuming that \eqref{MainAssumpt} holds true for all $x \in M$ 
there also holds that for any $\theta \in (0,1)$, 
$\Vert\mathcal{U}\Vert_{C^{2,\theta}} \le C$ for all $\mathcal{U} \in \mathcal{S}_4(\omega^\prime)$ 
and all $\omega^\prime \in K_0(\omega)$, where $C > 0$ does not depend on $\omega^\prime$ and $\mathcal{U}$, 
$\mathcal{S}_4(\omega^\prime)$ 
is as in \eqref{DefSp}, and $K_0(\omega)$ is as in \eqref{DefKOm0}.
\end{thm}

There are two consequences to Theorem \ref{Thm2}. We list them in points (i)-(ii) below. 
In point (i) 
we illustrate the phase compensation effect associated with \eqref{SWSyst}. There we always get 
existence and a priori bounds for all phases $\omega$ which are close to $m_0$. Point (ii) concerns the full range of phases 
when we assume $m_0$ is not too large. 

\medskip (i) Phase compensation in the critical case - {\it Assume $p = 4$ and $S_g > 0$ in $M$. 
Then there exists $\varepsilon > 0$ such that for any $m_0 -\varepsilon < \vert\omega\vert < m_0$, 
\eqref{SWSyst} possesses a smooth positive MPT solution. Moreover, for any $\theta \in (0,1)$, there exists $C > 0$ such that 
$\Vert\mathcal{U}\Vert_{C^{2,\theta}} \le C$ for all $\mathcal{U} \in \mathcal{S}_4(\omega)$ 
and all $m_0 -\varepsilon < \vert\omega\vert < m_0$.}

\medskip (iii) Full phase range in the critical case - {\it Assume $p = 4$ and $m_0^2 < \frac{1}{6}S_g$ in $M$. 
For any $\omega \in (-m_0,m_0)$, 
\eqref{SWSyst} possesses a smooth positive MPT solution. Moreover, for any $\theta \in (0,1)$, there exists $C > 0$ such that 
$\Vert\mathcal{U}\Vert_{C^{2,\theta}} \le C$ for all $\mathcal{U} \in \mathcal{S}_4(\omega)$ 
and all $\omega \in (-m_0,m_0)$.}

\medskip As an immediate consequence of the $C^{2,\theta}$-bounds in the above results we obtain phase stability for standing waves of 
the Klein-Gordon-Maxwell-Proca equations in electrostatic form. Standing waves for the Klein-Gordon-Maxwell-Proca equations in 
electrostatic form are written as $S = ue^{i\omega t}$ and they are coupled with a gauge potential $v$, where $(u,v)$ solves 
\eqref{SWSyst}. Roughly speaking, phase stability means that for any arbitrary sequence of standing waves 
$u_\alpha e^{i\omega_\alpha t}$, with gauge potentials $v_\alpha$, the convergence of the phases $\omega_\alpha$ in $\mathbb{R}$ 
implies the convergence of the amplitudes $u_\alpha$ and of the gauge potentials $v_\alpha$ in the $C^2$-topology. Phase stability 
prevents the existence of arbitrarily large amplitude standing waves.

\medskip High dimensional KGM systems in Coulomb gauge have been recently investigated by Rodnianski and Tao 
\cite{RodTao} and with special emphasis in $(1+4)$-dimensions by Klainerman and Tataru 
\cite{KlaTat} and Selberg \cite{Sel}. Electrostatic KGM systems in the three dimensional case have been investigated 
by several authors. Possible references on the physics side are by Benci and Fortunato \cite{BenFor0,BenFor}, Long \cite{Lon}, Long and Stuart \cite{LonStu}.
Blowing-up solutions to the electrostatic Schr\"odinger-Maxwell system, a cousin of the electrostaic KGM type systems 
that we consider here, have been constructed in D'Aprile and Wei \cite{AprWei1,AprWei2}.

\medskip We briefly discuss in Section \ref{PhysSec} the physics relevance of \eqref{SWSyst}. We prove our theorem in Sections 
\ref{ExistTheory} to \ref{Apriori2}. The existence part in the theorem is proved in Section \ref{ExistTheory}. The $C^{2,\theta}$-bound 
in the subcritical case is established in Section \ref{Apriori1}. The more delicate $C^{2,\theta}$-bound in the critical case is 
established in Sections \ref{Apriori2} . 
The phase compensation phenomenon in the theorem holds true thanks to the $4$-dimensional log effect $\mu^2 = o(\mu^2\log\mu)$ as 
$\mu \to 0$. 

\section{The physics origin of the system}\label{PhysSec}

The Klein-Gordon-Maxwell-Proca system
discussed in this work describes 
an interacting field theory model in theoretical physics. Most
electromagnetic phenomena are described by conventional
electrodynamics, which is a theory of the coupling of
electromagnetic fields to matter fields. Of prime importance for
particle physics is fermion electrodynamics in which matter is
represented by spinor fields. However one may have also boson
electrodynamics in which matter is described by integer spin or bosonic fields.
The simplest one is of course the complex scalar
field, describing spinless particles having electric charges $\pm q$. It gives rise to scalar electrodynamics, which describes in the
non-relativistic limit the superconductivity of metals at very low
temperatures. In the more general context of particle physics, a complex scalar
field $\psi$ may serve to describe scalar mesons in nuclear
matter interacting via a massive vector boson field $(\varphi,A)$. 

\medskip The interaction in this model is described by the minimum 
substitution rule \eqref{MinSubsRule} in a nonlinear Klein-Gordon Lagrangian. As for the external massive vector 
field it is governed by the Maxwell-Proca Lagrangian. 
More precisely, assuming for short that the manifold is orientable, we define the Lagrangian densities 
$\mathcal{L}_{NKG}$ and $\mathcal{L}_{MP}$ of $\psi$, $\varphi$, and $A$ by
\begin{equation}\label{EqtLagAct1}
\begin{split}
&\mathcal{L}_{NKG}(\psi,\varphi,A) = \frac{1}{2}\left\vert(\frac{\partial}{\partial t} + iq\varphi)\psi\right\vert^2 - \frac{1}{2}\left\vert(\nabla - iqA)\psi\right\vert^2
+ \frac{m_0^2}{2}\vert\psi\vert^2 - \frac{1}{p}\vert\psi\vert^{p}\hskip.1cm,\\
&\mathcal{L}_{MP}(\varphi,A) = \frac{1}{2}\left\vert\frac{\partial A}{\partial t} + \nabla\varphi\right\vert^2 - \frac{1}{2}\vert\nabla\times A\vert^2 
+ \frac{m_1^2}{2}\vert\varphi\vert^2 - \frac{m_1^2}{2}\vert A\vert^2\hskip.1cm ,
\end{split}
\end{equation}
where $\nabla\times = \star d$, $\star$ is the Hodge dual, $\psi$ represents 
the matter complex scalar field, $m_0$ its mass, $q$ its charge, $(\varphi,A)$ the electromagnetic vector field, and $m_1$ its mass. 
It can be noted that $\Vert(\varphi,A)\Vert_L^2 = \vert\varphi\vert^2-\vert A\vert^2$ is the square of the Lorentz norm of $(\varphi,A)$ with respect to the 
Lorentz metric $\hbox{diag}(1,-1,\dots,-1)$. 
The total action functional for $\psi$, $\phi$, and $A$ is then given by
\begin{equation}\label{EqtLagAct2}
\mathcal{S}(\psi,\varphi,A) = \int\int \left(\mathcal{L}_{NKG} + \mathcal{L}_{MP}\right) dv_gdt\hskip.1cm .
\end{equation}
Writing $\psi$ in polar form as $\psi(x,t) = u(x,t)e^{iS(x,t)}$, 
taking the variation of $\mathcal{S}$ with respect to $u$, $S$, $\varphi$, and $A$,
we get four equations which are written as
\begin{equation}\label{EqtLagAct5}
\begin{cases}
\frac{\partial^2u}{\partial t^2} + \Delta_gu + m_0^2u = u^{p-1} + \left(\left(\frac{\partial S}{\partial t} + q\varphi\right)^2 - \vert\nabla S-qA\vert^2\right)u\\
\frac{\partial}{\partial t}\left(\left(\frac{\partial S}{\partial t} + q\varphi\right)u^2\right) - \nabla.\left(\left(\nabla S - qA\right)u^2\right) = 0\\
-\nabla.\left(\frac{\partial A}{\partial t} + \nabla\varphi\right) + m_1^2\varphi + q\left(\frac{\partial S}{\partial t} + q\varphi\right)u^2 = 0\\
\overline{\Delta}_gA + \frac{\partial}{\partial t}\left(\frac{\partial A}{\partial t} + \nabla\varphi\right) + m_1^2A = q\left(\nabla S - qA\right)u^2\hskip.1cm ,
\end{cases}
\end{equation}
where $\Delta_g = -\hbox{div}_g\nabla$ is the Laplace-Beltrami operator, 
$\overline{\Delta}_g = \delta d$ is half the Laplacian acting on forms, and $\delta$ is the codifferential. We refer to this system 
as a nonlinear Klein-Gordon-Maxwell-Proca system. When $n = 3$, 
$\overline{\Delta}_gA = \nabla\times(\nabla\times A)$ and if we let
\begin{equation}\label{EqtLagAct6}
\begin{split}
&E = - \left(\frac{\partial A}{\partial t} + \nabla\varphi\right)\hskip.1cm ,\hskip.1cm H = \nabla\times A\hskip.1cm ,\\
&\rho = - \left(\frac{\partial S}{\partial t} + q\varphi\right)qu^2
\hskip.1cm ,\hskip.1cm\hbox{and}\hskip.1cm j = \left(\nabla S - qA\right)qu^2\hskip.1cm ,
\end{split}
\end{equation}
then the two last equations in \eqref{EqtLagAct5} give rise to the first pair of the Maxwell-Proca equations 
\eqref{MaxwellProca} with $\epsilon_0 = \mu_0 = 1$ (units are chosen such that $c = 1$) 
and $\mu^2 = m_1^2$, while the two first equations in \eqref{EqtLagAct6} give rise to the second pair of the 
equations. The first equation in \eqref{EqtLagAct5} gives rise to the nonlinear Klein-Gordon matter equation. The 
second equation in \eqref{EqtLagAct5} gives rise to the charge continuity equation
$\frac{\partial\rho}{\partial t} + \nabla.j = 0$ which, thanks to \eqref{MaxwellProca}, is equivalent to the Lorentz condition
$\nabla.A + \frac{\partial\varphi}{\partial t} = 0$.

\medskip We assume in what follows that $u(x,t) = u(x)$ does not depend on $t$, $S(x,t) = \omega t$ does not depend on $x$, 
and $\varphi(x,t) = \varphi(x)$, $A(x,t) = A(x)$ do not depend on $t$. In other words, we look for standing waves solutions of 
\eqref{EqtLagAct5} and assume that we are in the static case of the system where $(\varphi,A)$ depends on the sole 
spatial variable. By the fourth equation in \eqref{EqtLagAct5} 
we then get that
$$\overline{\Delta}_gA + (q^2u^2+m_1^2)A = 0\hskip.1cm .$$
This clearly implies that, and is equivalent to, $A \equiv 0$ since 
$$\int (\overline{\Delta}_gA,A) = \int\vert dA\vert^2\hskip.1cm .$$
As a remark, assuming that 
$A \equiv 0$, the Lorentz condition for the external Proca field $(\varphi,A)$ would make $\varphi$ dependent on the sole spatial variables. 
 As for the second equation 
in \eqref{EqtLagAct5} it reduces to $\frac{\partial^2S}{\partial t^2} = 0$. It is automatically satisfied when $S(t) = \omega t$, 
and we are thus left with the 
first and third equations in \eqref{EqtLagAct5}. Letting $S = -\omega t$, and $\varphi = \omega v$, these equations are rewritten 
as
\begin{equation}\label{EqtLagAct11}
\begin{cases}
\Delta_gu + m_0^2u = u^{p-1} + (q\varphi-\omega)^2u    \\
\Delta_g\varphi + m_1^2\varphi + q(q\varphi-\omega)u^2 = 0\hskip.1cm .
\end{cases}
\end{equation}
In particular, letting $\varphi = \omega v$, in \eqref{EqtLagAct11}, we recover our original system \eqref{SWSyst}. 
In other words, our original system \eqref{SWSyst} corresponds to looking for standing waves solutions of 
the Klein-Gordon-Maxwell-Proca system \eqref{EqtLagAct5} in static form.

\section{Existence Theory}\label{ExistTheory}

We prove the existence part in Theorems \ref{Thm1} and \ref{Thm2} and look for solutions with a 
special variational structure. Formally, solutions of \eqref{SWSyst} are critical points of the functional $S$ defined by
\begin{equation}\label{FullFunct}
\begin{split}
S(u,v) & = \frac{1}{2}\int_M\vert\nabla u\vert^2dv_g - \frac{\omega^2}{2}\int_M\vert\nabla v\vert^2dv_g + \frac{m_0^2}{2}\int_Mu^2dv_g\\
&-\frac{\omega^2m_1^2}{2}\int_Mv^2dv_g- \frac{1}{p}\int_Mu^pdv_g - \frac{\omega^2}{2}\int_Mu^2(1-qv)^2dv_g\hskip.1cm .
\end{split}
\end{equation}
The functional $S$ is strongly indefinite because of the competition between $u$ and $v$. 
Following a very nice idea going back to Benci-Fortunato \cite{BenFor0}, 
we introduce the auxiliary functional $\Phi$ given by 
\begin{equation}\label{DefPhi}
\Delta_g\Phi(u) + (m_1^2+q^2u^2)\Phi(u) = qu^2
\hskip.1cm ,
\end{equation}
and then consider that $u$ in \eqref{SWSyst} can be seen as a critical point of
\begin{equation}\label{DefFct}
\begin{split}
I_p(u) &= \frac{1}{2}\int_M\vert\nabla u\vert^2dv_g + \frac{m_0^2}{2}\int_Mu^2dv_g - \frac{1}{p}\int_M(u^+)^pdv_g\\
&- \frac{\omega^2}{2}\int_M\left(1-q\Phi(u)\right)u^2dv_g\hskip.1cm ,
\end{split}
\end{equation}
where $u^+ = \max(u,0)$. We explicitly define MPT solutions to be solutions we obtain from $I_p$ by the mountain 
pass lemma from $0$ to $u_1$ with $\Vert u_1\Vert_{L^p}^p$ being as large as we want with respect to $\Vert u_1\Vert_{H^1}^2$. 
Let $\Psi: H^1(M) \to \mathbb{R}$ be defined by
\begin{equation}\label{DefPsi}
\Psi(u) = \frac{1}{2}\int_M\left(1-q\Phi(u)\right)u^2dv_g\hskip.1cm .
\end{equation}
The following lemma establishes the existence and differentiability of $\Phi$, as well as the $C^1$-smoothness of $\Psi$. 
Equation \eqref{DefPhi} is critical when $n = 4$ 
because of the term $u^2\Phi(u)$.

\begin{lem}\label{Affine4Dim} Let $(M,g)$ be a smooth compact Riemannian $4$-manifold and $q > 0$. There 
exists $\Phi: H^1(M) \to H^1(M)$ such that \eqref{DefPhi} holds true and 
$0 \le \Phi(u) \le \frac{1}{q}$ for all $u \in H^1(M)$. Moreover, $\Phi$ is locally Lipschitz and differentiable. Its differential $D\Phi(u) = V_u$ at $u$ is given by
\begin{equation}\label{EqtDiffPhiLem1}
\Delta_g V_u(\varphi) + (m_1^2+q^2u^2)V_u(\varphi) = 2qu\left(1-q\Phi(u)\right)\varphi
\end{equation}
for all $\varphi \in H^1(M)$. The functional $\Psi: H^1(M) \to \mathbb{R}$ defined in \eqref{DefPsi}
is $C^1$ in $H^1(M)$ and 
\begin{equation}\label{DiffPsi}
D\Psi(u).(\varphi) = \int_M\left(1-q\Phi(u)\right)^2u\varphi dv_g
\end{equation}
for all $u, \varphi \in H^1(M)$.
\end{lem}

\begin{proof}[Proof of Lemma \ref{Affine4Dim}]  We briefly sketch the proof. Let $u \in H^1$ and $H_u: H^1 \to \mathbb{R}$ be defined by
$$H_u(\varphi) = \int_M\vert\nabla\varphi\vert^2dv_g + \int_M(m_1^2+q^2u^2)\varphi^2dv_g
\hskip.1cm .$$
The functional is well defined since $H^1 \subset L^4$. Letting $\Phi(0) = 0$ we can assume that $u \not\equiv 0$. Let
$$\mu = \inf_{u \in H^1, \int u^2\varphi = 1}H_u(\varphi)
\hskip.1cm .$$
By standard minimization arguments there exists $\varphi \in H^1(M)$ such that $\int_Mu^2\varphi dv_g = 1$ and $H_u(\varphi) = \mu$. In particular, 
$\mu > 0$. Letting $\Phi(u) = \frac{q}{\mu}\varphi$ we get that $\Phi(u)$ solves \eqref{DefPhi} in $H^1$. 
It is easily seen that $\Phi(u)$ is 
unique. By the maximum principle, 
$\Phi(u) \ge 0$. Noting that
$$\Delta_g\left(\frac{1}{q}-\Phi(u)\right) + (m_1^2+q^2)u^2\left(\frac{1}{q}-\Phi(u)\right) \ge 0$$
it also follows from the maximum principle that $\Phi(u) \le \frac{1}{q}$. Now we let 
$u, v \in  H^1(M)$. We have that
$$\Delta_g\left(\Phi(v)-\Phi(u)\right) + (m_1^2+q^2)u^2\left(\Phi(v)-\Phi(u)\right) = q(v^2-u^2)\left(1-q\Phi(v)\right)
\hskip.1cm .$$
Multiplying the equation by $\Phi(v)-\Phi(u)$, integrating over $M$, and by the Sobolev emedding theorem, 
we get that
\begin{equation}\label{ContEqt1}
\Vert\Phi(v)-\Phi(u)\Vert_{H^1} \le C\left(\Vert u\Vert_{H^1} + \Vert v\Vert_{H^1}\right)\Vert v-u\Vert_{H^1}
\hskip.1cm .
\end{equation}
In particular, $\Phi$ is locally Lipschitz continuous. We can prove the existence of $V_u(\varphi)$ 
in \eqref{EqtDiffPhiLem1} as when proving the existence of $\Phi(u)$. 
Writing the equation satisfied by $\Phi(u+\varphi)-\Phi(u)-V_u(\varphi)$, multiplying the equation by $\Phi(u+\varphi)-\Phi(u)-V_u(\varphi)$ and integrating 
over $M$, we get that
$$\Vert\Phi(u+\varphi)-\Phi(u)-V_u(\varphi)\Vert_{H^1}
\le C \Vert\varphi\Vert_{H^1}
\left(\Vert\varphi\Vert_{H^1} + \Vert u\Vert_{H^1}\Vert\Phi(u+\varphi)-\Phi(u)\Vert_{H^1}\right)
$$
Then the differentiability of $\Phi$ follows from the continuity of $\Phi$. In particular, 
$\Psi$ is differentiable. 
By \eqref{DefPhi},
$$\Psi(u) = \frac{1}{2}\int_M\left(\vert\nabla\Phi(u)\vert^2+m_1^2\Phi(u)^2\right)dv_g 
+ \frac{1}{2}\int_M\left(1-q\Phi(u)\right)^2u^2dv_g\hskip.1cm ,$$
and we also have that $\frac{\partial H}{\partial\Phi}\left(u,\Phi(u)\right) = 0$, where 
$H(u,\Phi) = \frac{1}{2}H_u(\Phi)
- q\int_Mu^2\Phi dv_g$.
Noting that
$$\Psi(u) = H\left(u,\Phi(u)\right) + \frac{1}{2}\int_Mu^2dv_g\hskip.1cm ,$$
we get that \eqref{DiffPsi} holds true. The continuity of $D\Psi$ can be proved directly from \eqref{DiffPsi} and the continuity of $\Phi$. 
This ends the proof of the lemma.
\end{proof}

Now we prove the subcritical existence of Theorem \ref{Thm1}. We proceed by applying the mountain pass lemma to the functional 
$I_p$ in \eqref{DefFct}.

\begin{proof}[Proof of existence in Theorem \ref{Thm1}] By Lemma \ref{Affine4Dim}, $I_p$ is $C^1$ in $H^1$. Let $u_0 \in H^1$ such that 
$u_0^+ \not\equiv 0$. There holds $I_p(0) = 0$ and $I_p(tu_0) \to -\infty$ as $t \to +\infty$ since $p > 2$. 
Since $0 \le \Phi(u) \le \frac{1}{q}$ for all $u$, we also have that 
\begin{eqnarray*} I_p(u) 
&\ge& \frac{1}{2}\left(\int_M\vert\nabla u\vert^2dv_g + (m_0^2-\omega^2)\int_Mu^2dv_g\right) - \frac{1}{p}\int_M\vert u\vert^pdv_g\\
&\ge& C_1 \Vert u\Vert_{H^1}^2 - C_2\Vert u\Vert_{H^1}^p
\end{eqnarray*}
for all $u \in H^1$, where $C_1, C_2 > 0$ do not depend on $u$. In particular, there exist $\delta, C > 0$ such that $I_p(u) \ge C$ for all $u \in H^1$ such that 
$\Vert u\Vert_{H^1} = \delta$. Let $T_0 = T_0(u_0)$, $T_0 \gg 1$, be such that $I_p(T_0u_0) < 0$, and $c_p = c_p(u_0)$ be given by 
\begin{equation}\label{Defcp}
c_p = \inf_{P\in\mathcal{P}}\max_{u \in P}I_p(u)
\hskip.1cm ,
\end{equation}
where $\mathcal{P}$ is the class of continuous paths joining $0$ to $T_0u_0$. According to the above we can apply the mountain pass lemma and we get the 
existence of a sequence $(u_\alpha)_\alpha$ in $H^1$ such that $I_p(u_\alpha) \to c_p$ and $DI_p(u_\alpha) \to 0$ as $\alpha \to +\infty$. Writing that 
$I_p(u_\alpha) = c_p + o(1)$ and that $DI_p(u_\alpha).(u_\alpha) = o(\Vert u_\alpha\Vert_{H^1})$, we get by Lemma \ref{Affine4Dim} that
\begin{equation}\label{SubCritEqt1}
\begin{split}
&\frac{1}{2}\int_M\left(\vert\nabla u_\alpha\vert^2+m_0^2u_\alpha^2\right)dv_g\\
&\hskip.4cm = \frac{1}{p}\int_M(u_\alpha^+)^pdv_g + c_p +  
\frac{\omega^2}{2}\int_M\left(1-q\Phi(u_\alpha)\right)u_\alpha^2dv_g + o(1)\\
&\frac{1}{2}\int_M\left(\vert\nabla u_\alpha\vert^2+m_0^2u_\alpha^2\right)dv_g\\
&\hskip.4cm = \frac{1}{2}\int_M(u_\alpha^+)^pdv_g 
+ \frac{\omega^2}{2}\int_M\left(1-q\Phi(u_\alpha)\right)^2u_\alpha^2dv_g + o\left(\Vert u_\alpha\Vert_{H^1}\right)
\end{split}
\end{equation}
for all $\alpha$. 
Writing that $DI_p(u_\alpha).(u_\alpha^-) = o(\Vert u_\alpha^-\Vert_{H^1})$ we get that $u_\alpha^- \to 0$ 
in $H^1$ as $\alpha \to +\infty$. By \eqref{SubCritEqt1} we then get that 
$(u_\alpha)_\alpha$ is bounded in $H^1$. In particular, there exists $u_p \in H^1(M)$ such that, up to passing to a subsequence, 

\medskip (i) $u_\alpha \rightharpoonup u_p$ weakly in $H^1$, 

\medskip (ii) $u_\alpha \to u_p$ in $L^p$,

\medskip\noindent and $u_\alpha \to u_p$ a.e. as $\alpha \to +\infty$. Substracting one equation to another in \eqref{SubCritEqt1}, 
letting $\alpha \to +\infty$, and since $c_p \not= 0$, we get that $u_p \not\equiv 0$. Writing the equation satisfied by 
$\Phi(u_\alpha)-\Phi(u_p)$, multiplying the equation by $\Phi(u_\alpha)-\Phi(u_p)$ and integrating over $M$, we get that 
\begin{equation}\label{ConvStepSubCase}
\Phi(u_\alpha) \to \Phi(u_p)\hskip.2cm\hbox{in}\hskip.1cm H^1
\end{equation}
as $\alpha \to +\infty$. 
Now we let $\varphi \in H^1$. There holds $DI_p(u_\alpha).(\varphi) = o(1)$. Hence, 
by Lemma \ref{Affine4Dim},
\begin{equation}\label{SubCritEqt5}
\begin{split}
&\int_M\nabla u_\alpha\nabla\varphi dv_g + m_0^2\int_Mu_\alpha\varphi dv_g\\
&= \int_M(u_\alpha^+)^{p-1}\varphi dv_g + \omega^2\int_M\left(1-q\Phi(u_\alpha)\right)^2u_\alpha\varphi dv_g + o(1)
\hskip.1cm .
\end{split}
\end{equation}
Letting $\alpha \to +\infty$ in \eqref{SubCritEqt5} we then get by \eqref{ConvStepSubCase} that 
$$\Delta_gu_p + m_0^2u_p = (u_p^+)^{p-1} + \omega^2\left(1-q\Phi(u_p)\right)^2u_p$$
in $H^1$. Multiplying the equation by $u_p^-$ and integrating over $M$, it follows 
that $u_p^- \equiv 0$. In particular, $u_p \ge 0$, $u_p \not\equiv 0$, and 
\begin{equation}\label{SubCritEqt7}
\Delta_gu_p + m_0^2u_p = u_p^{p-1} + \omega^2\left(1-q\Phi(u_p)\right)^2u_p
\end{equation}
in $H^1$. By regularity results we get from \eqref{SubCritEqt7} that $u_p \in H^{2,s}$ for all $s$. Then, 
by regularity results, $\Phi(u_p) \in H^{2,s}$ 
for all $s$. 
By the Sobolev embedding theorem, regularity theory, and the maximum principle, it follows that $u_p, \Phi(u_p) 
\in C^2(M)$ and that $u_p, \Phi(u_p) > 0$ in $M$. Letting $u = u_p$ and 
$v = \Phi(u_p)$, this proves the existence part in Theorem \ref{Thm1}. 
\end{proof}

An additional information we obtain is that $u_p$ realizes $c_p$. Indeed, since $u_p \ge 0$, 
$u_\alpha^- \to 0$ in $H^1$, and $\Phi(u_\alpha)\to\Phi(u_p)$ 
in $H^1$, there holds that 
\begin{eqnarray*}
&&\int_M(u_\alpha^+)^pdv_g \to \int_Mu_p^pdv_g
\hskip.1cm ,\hskip.1cm\hbox{and}\\
&&\int_M\left(1-q\Phi(u_\alpha)\right)^2u_\alpha^2dv_g \to \int_M\left(1-q\Phi(u_p)\right)^2u_p^2dv_g
\end{eqnarray*}
as $\alpha \to +\infty$. The second equation in \eqref{SubCritEqt1} together with \eqref{SubCritEqt7} then give that 
$$\int_M\vert\nabla u_\alpha\vert^2dv_g \to \int_M\vert\nabla u_p\vert^2dv_g
\hskip.1cm .$$
It follows that $u_\alpha \to u_p$ in $H^1$ as $\alpha \to +\infty$. By the first equation in \eqref{SubCritEqt1} we then get 
that $I_p(u_p) = c_p$. In other words, $c_p$ is realized by $u_p$.
Now, given $x_0 \in M$ and $\rho_0 > 0$ small, sufficiently small, 
we define $u_\varepsilon$ by
\begin{equation}\label{DefuEpsHighDim}
\begin{cases}
u_\varepsilon(x) = \frac{\varepsilon}{\varepsilon^2+r^2}-\frac{\varepsilon}{\varepsilon^2+\rho_0^2}
\hskip.2cm\hbox{if}~r \le \rho_0\hskip.1cm ,\\
u_\varepsilon(x) = 0\hskip.2cm\hbox{if}\hskip.1cm r \ge \rho_0\hskip.1cm ,
\end{cases}
\end{equation}
where $r = d_g(x_0,x)$. Then, see Aubin \cite{Aub}, for any $\lambda \in \mathbb{R}$,
\begin{equation}\label{TestComputaHighDim}
J_\lambda(u_\varepsilon)= \frac{1}{K_4^2}
\left(1 + C\left(\frac{1}{6}S_g(x_0)-\lambda\right)\varepsilon^2\ln\varepsilon + o(\varepsilon^2\ln\varepsilon)\right)\hskip.1cm ,
\end{equation}
where
$$J_\lambda(u_\varepsilon) = \frac{\int_M\left(\vert\nabla u_\varepsilon\vert^2+\lambda u_\varepsilon^2\right)dv_g}{\left(\int_Mu_\varepsilon^4dv_g\right)^{1/2}}
\hskip.1cm ,$$
and $C > 0$ is independent of $\alpha$. Also there holds
\begin{equation}\label{QuotSch2bis}
\begin{split}
&\int_Mu_\varepsilon^4dv_g = \int_{\mathbb{R}^3}\left(\frac{1}{1+\vert x\vert^2}\right)^4dx + o(1)\hskip.1cm ,\\
&\int_M\vert\nabla u_\varepsilon\vert^2dv_g = 8\int_Mu_\varepsilon^4dv_g + o(1)\hskip.1cm .
\end{split}
\end{equation}
In what follows we prove the existence part of Theorem \ref{Thm2}.

\begin{proof}[Proof of existence in Theorem \ref{Thm2}] As a preliminary remark, 
by standard arguments such as developed in Aubin \cite{Aub} and Br\'ezis and Nirenberg \cite{BreNir}, we just need to prove that we can chose 
$u_0 \in H^1$, $u_0^+ \not\equiv 0$, such that
\begin{equation}\label{AubinTypeIneqt}
\delta_0 \le c_p \le \frac{1}{4K_4^4}-\delta_0
\end{equation}
for all $p \in (4-\varepsilon,4)$ and some $\varepsilon, \delta_0 > 0$, where $c_p = c_p(u_0)$ is as in \eqref{Defcp}. Now 
we assume that \eqref{MainAssumpt} holds true for some $x \in M$, in particular for $x \in M$ where $S_g$ is maximum. 
We let $x_0 \in M$ be such that $S_g$ is maximum at $x_0$, and 
$(t_\varepsilon)_\varepsilon$ be any family 
of positive real numbers such that the $t_\varepsilon$'s are 
bounded. The first estimate we prove is that
\begin{equation}\label{Dim4KeyEst}
\int_M\Phi(t_\varepsilon u_ \varepsilon)u_ \varepsilon ^2dv_g = O\left(\varepsilon^2\right)
\hskip.1cm ,
\end{equation}
where the $u_ \varepsilon $'s are as in \eqref{DefuEpsHighDim}. By definition,
\begin{equation}\label{EstDim4Eqt1}
\Delta_g\Phi(t_\varepsilon u_\varepsilon) + (m_1^2+q^2)t_\varepsilon^2u_\varepsilon^2\Phi(t_\varepsilon u_\varepsilon) = qt_\varepsilon^2u_\varepsilon^2
\hskip.1cm .
\end{equation}
Multiplying \eqref{EstDim4Eqt1} by $\Phi(t_\varepsilon u_\varepsilon)$ and integrating over $M$ we get by H\"older's inequalities that
\begin{eqnarray*} \Vert\Phi(t_\varepsilon u_\varepsilon)\Vert^2_{H^1}
&=& qt_\varepsilon^2\int_Mu_\varepsilon^2\Phi(t_\varepsilon u_\varepsilon)dv_g\\
&\le& C \left(\int_Mu_\varepsilon^{8/3}dv_g\right)^{3/4}\Vert\Phi(t_\varepsilon u_\varepsilon)\Vert_{L^4}
\end{eqnarray*}
and it follows from the Sobolev inequality that
\begin{equation}\label{EstDim4Eqt2}
\Vert\Phi(t_\varepsilon u_\varepsilon)\Vert_{H^1} \le C \left(\int_Mu_\varepsilon^{8/3}dv_g\right)^{3/4}
\hskip.1cm .
\end{equation}
Then, by \eqref{EstDim4Eqt2},
\begin{equation}\label{EstDim4Eqt3}
\begin{split}
\int_M\Phi(t_\varepsilon u_ \varepsilon)u_ \varepsilon ^2dv_g 
&\le  C \left(\int_Mu_\varepsilon^{8/3}dv_g\right)^{3/4}\Vert\Phi(t_\varepsilon u_\varepsilon)\Vert_{L^4}\\
&\le C \left(\int_Mu_\varepsilon^{8/3}dv_g\right)^{3/2}
\hskip.1cm .
\end{split}
\end{equation}
There holds,
\begin{equation}\label{EstDim4Eqt4}
\begin{split}
\int_Mu_\varepsilon^{8/3}dv_g 
&\le \omega_3\int_0^{\rho_0}\left(\frac{\varepsilon}{\varepsilon^2 + r^2}\right)^{8/3}r^3dr\\
&= \omega_3\varepsilon^{4/3}\int_0^{\rho_0/\varepsilon}\left(\frac{1}{1 + r^2}\right)^{8/3}r^3dr\\
&= O(\varepsilon^{4/3})\hskip.1cm .
\end{split}
\end{equation}
By \eqref{EstDim4Eqt3} and \eqref{EstDim4Eqt4}, this proves \eqref{Dim4KeyEst}. 
Let $(\varepsilon_\alpha)_\alpha$ be a sequence of positive real 
numbers such that $\varepsilon_\alpha \to 0$ as $\alpha \to +\infty$, $u_\alpha = u_{\varepsilon_\alpha}$, and $\mathcal{F}_4$ be the functional defined in $H^1$ 
by 
\begin{equation}\label{DefF6Proof1bis}
\mathcal{F}_4(u) = \frac{1}{2}\int_M\vert\nabla u\vert^2dv_g + \frac{1}{2}(m_0^2-\omega^2)\int_Mu^2dv_g - 
\frac{1}{4}\int_M\vert u\vert^4dv_g\hskip.1cm .
\end{equation}
By \eqref{QuotSch2bis}, there exists $T_0 \gg 1$ such that $I_4(T_0u_\alpha) < 0$ for all $\alpha \gg 1$. There also holds that
\begin{eqnarray*} \max_{0\le t\le T_0}I_4(tu_\alpha) 
&\le& \max_{0\le t\le T_0}\mathcal{F}_4(tu_\alpha) + CT_0^2\max_{0\le t\le T_0}\int_M\Phi(tu_\alpha)u_\alpha^2dv_g\\
&\le& \frac{1}{4}J_\lambda(u_\alpha)^2 + CT_0^2\max_{0\le t\le T_0}\int_M\Phi(tu_\alpha)u_\alpha^2dv_g
\end{eqnarray*}
for all $\alpha$, where $\lambda = m_0^2-\omega^2$. By \eqref{TestComputaHighDim} and \eqref{Dim4KeyEst} we thus get that
$$\max_{0\le t\le T_0}I_4(tu_\alpha) \le \frac{1}{K_4^4}
\left(1 + C\left(\frac{1}{6}S_g(x_0)-m_0^2+\omega^2\right)\varepsilon_\alpha^2\ln\varepsilon_\alpha + o(\varepsilon_\alpha^2\ln\varepsilon_\alpha)\right)\hskip.1cm ,$$
where $C > 0$ is independent of $\alpha$. By assumption the $\varepsilon_\alpha^2\ln\varepsilon_\alpha$ coefficient is positive. 
Let $u_0 = u_\alpha$, where $\alpha \gg 1$ is sufficiently large such that
$$\max_{0\le t\le T_0}I_4(tu_\alpha) \le \frac{1}{4K_4^4} - \delta_0$$
for some $\delta_0 > 0$. Since $u_0$ is now fixed, we can write that
\begin{equation}\label{CompFctsProof1Eqt2bis}
\max_{0\le t\le T_0}I_p(tu_0) \le (1+\delta_\varepsilon)\max_{0\le t\le T_0}I_4(tu_0)
\end{equation}
for all $p \in (4-\varepsilon,4)$, where $\delta_\varepsilon >0$ is such that $\delta_\varepsilon \to 0$ 
as $\varepsilon \to 0$. Noting that 
\begin{eqnarray*}
I_p(u) 
&\ge& \frac{1}{2}\int_M\left(\vert\nabla u\vert^2+(m_0^2-\omega^2)u^2\right)dv_g - \frac{1}{p}\int_M\vert u\vert^pdv_g\hskip.1cm ,\\
&\ge& C_1\Vert u\Vert_{H^1}^2 - C_2\Vert u\Vert_{H^1}^p
\end{eqnarray*}
where $C_1, C_2 > 0$ are independent of $u$, there holds that there exist $\delta_1, \delta_2 > 0$ such that $\delta_1, \delta_2$ are as small as we want, and 
$I_p(u) \ge \delta_2$ for all $u$ such that $\Vert u\Vert_{H^1} = \delta_1$. As a 
conclusion, there exist $\delta_0 > 0$ and $\varepsilon > 0$ such that \eqref{AubinTypeIneqt} 
holds true 
for all $p \in (4-\varepsilon,4)$. This ends the proof of the existence part in Theorem \ref{Thm2}.
\end{proof}

There are always constant solutions to \eqref{SWSyst}. By \eqref{AubinTypeIneqt} the MPT solutions we obtain are distinct from these 
constant solutions in several situations, e.g. like on $S^1(T)\times S^3$ for $T\gg1$ when $m_1^2/q \ll 1$.

\section{A priori bounds in the subcritical case}\label{Apriori1}

We prove the uniform bounds in the subcritical case of Theorem \ref{Thm1}. In what follows $p \in (2,4)$.

\begin{proof}[Proof of the uniform bounds in Theorem \ref{Thm1}] Let $(\omega_\alpha)_\alpha$ be a sequence in $(-m_0,m_0)$ such that 
$\omega_\alpha \to \omega$ as $\alpha \to +\infty$ for some $\omega \in [-m_0,m_0]$. Also let 
$p \in (2,4)$ and $\bigl((u_\alpha,v_\alpha)\bigr)_\alpha$ be a sequence 
of smooth positive solutions of \eqref{SWSyst} with phases $\omega_\alpha$. Then,
\begin{equation}\label{SWSystAlpha}
\begin{cases}
\Delta_gu_\alpha + m_0^2u_\alpha = u_\alpha^{p-1} + \omega_\alpha^2\left(qv_\alpha-1\right)^2u_\alpha\\
\Delta_gv_\alpha + \left(m_1^2+q^2u_\alpha^2\right)v_\alpha = qu_\alpha^2
\end{cases}
\end{equation}
for all $\alpha$. By the second equation in \eqref{SWSystAlpha}, 
$0 \le v_\alpha \le \frac{1}{q}$ for all $\alpha$. Assume by contradiction that
\begin{equation}\label{ContrAssumptSubCptness1}
\max_Mu_\alpha \to +\infty
\end{equation}
as $\alpha \to +\infty$. Let $x_\alpha \in M$ and $\mu_\alpha > 0$ be given by
$$u_\alpha(x_\alpha) = \max_Mu_\alpha = \mu_\alpha^{-2/(p-2)}\hskip.1cm .$$
By \eqref{ContrAssumptSubCptness1}, $\mu_\alpha \to 0$ as $\alpha \to +\infty$. Define $\tilde u_\alpha$ by
$$\tilde u_\alpha(x) = \mu_\alpha^{\frac{2}{p-2}}u_\alpha\left(\exp_{x_\alpha}(\mu_\alpha x)\right)$$
and $g_\alpha$ by $g_\alpha(x) = \left(\exp_{x_\alpha}^\star g\right)(\mu_\alpha x)$ for $x \in B_0(\delta\mu_\alpha^{-1})$, 
where $\delta > 0$ is small. Since $\mu_\alpha \to 0$, we get that $g_\alpha \to \xi$ in $C^2_{loc}(\mathbb{R}^3)$ as $\alpha \to +\infty$. 
Moreover, by \eqref{SWSystAlpha},
\begin{equation}\label{EqtTildeUAlpha}
\Delta_{g_\alpha}\tilde u_\alpha + \mu_\alpha^2m_0^2\tilde u_\alpha = \tilde u_\alpha^{p-1} 
+ \omega_\alpha^2\mu_\alpha^2\left(q\hat v_\alpha - 1\right)^2\tilde u_\alpha\hskip.1cm ,
\end{equation}
where $\hat v_\alpha$ is given by $\hat v_\alpha(x) = v_\alpha\left(\exp_{x_\alpha}(\mu_\alpha x)\right)$. 
We have $\tilde u_\alpha(0) = 1$ and $0 \le \tilde u_\alpha \le 1$. By \eqref{EqtTildeUAlpha} and standard elliptic theory 
arguments, we can write that, after passing to a subsequence, $\tilde u_\alpha \to u$ in $C^{1,\theta}_{loc}(\mathbb{R}^4)$ 
as $\alpha \to +\infty$, where $u$ is such that $u(0) = 1$ and $0 \le u \le 1$. Then
$$\Delta u = u^{p-1}$$
in $\mathbb{R}^4$, where $\Delta$ is the Euclidean Laplacian. It follows that $u$ is actually smooth and positive, and, 
since $2 < p < 4$, we get 
a contradiction with the Liouville result of Gidas and Spruck \cite{GidSpr}. As a conclusion, \eqref{ContrAssumptSubCptness1} is 
not possible and there exists $C > 0$ such that
\begin{equation}\label{ConclContr1}
u_\alpha + v_\alpha \le C
\end{equation}
in $M$ for all $\alpha$. Coming back to \eqref{SWSystAlpha} it follows that the sequences $(u_\alpha)_\alpha$ and $(v_\alpha)_\alpha$ 
are actually bounded in $H^{2,s}$ for all $s$. Pushing one step further the regularity argument they turn out to be bounded in $H^{3,s}$ for all $s$, 
and by the Sobolev embedding theorem we get that they are also 
bounded in $C^{2,\theta}$, $0 < \theta < 1$. This ends the proof of the uniform bounds in Theorem \ref{Thm1} when $p \in (2,4)$. 
\end{proof}

If we assume that 
$\omega_\alpha \to \omega$ as $\alpha \to +\infty$ for some $\omega \in (-m_0,m_0)$, 
$p \in (2,4]$, and $u_\alpha \to u$ and $v_\alpha \to v$ 
in $C^2$ as $\alpha \to +\infty$, then $u > 0$, $v > 0$, and $u, v$ are smooth solutions of \eqref{SWSyst}.
Indeed, given $\varepsilon > 0$ sufficiently small, since $m_0^2-\omega^2 > 0$, 
$\Delta_g + (m_0^2-\omega^2-\varepsilon)$ is coercive. There holds that $0 \le v_\alpha \le \frac{1}{q}$ for all $\alpha$. 
In particular, by \eqref{SWSystAlpha} and the Sobolev inequality, for any $\alpha \gg 1$ sufficiently large,
\begin{eqnarray*} 
&&\int_M\left(\vert\nabla u_\alpha\vert^2+ \left(m_0^2-\omega^2-\varepsilon\right)u_\alpha^2\right)dv_g\\
&&\le \int_M\vert\nabla u_\alpha\vert^2dv_g+ m_0^2\int_Mu_\alpha^2dv_g - \omega_\alpha^2\int_M(qv_\alpha^2-1)^2u_\alpha^2dv_g\\
&&= \int_Mu_\alpha^pdv_g
\le C\left(\int_M\left(\vert\nabla u_\alpha\vert^2+ \left(m_0^2-\omega^2-\varepsilon\right)u_\alpha^2\right)dv_g\right)^{p/2}
\end{eqnarray*}
for some $C > 0$ independent of $\alpha$. This implies $u > 0$ and then $v > 0$. 
Obviously the positivity of $u$ and $v$ does not hold anymore if we allow 
$\omega^2 = m_0^2$. Let $(\varepsilon_\alpha)_\alpha$ be a sequence of positive real numbers such that $\varepsilon_\alpha \to 0$ as 
$\alpha \to +\infty$. Let $u_\alpha = \varepsilon_\alpha$ and
$$v_\alpha = \frac{q\varepsilon_\alpha^2}{m_1^2 + q^2\varepsilon_\alpha^2}\hskip.1cm .$$
Then $u_\alpha \to 0$ and 
$v_\alpha \to 0$ in $C^2$ as $\alpha \to +\infty$, and we do have that $(u_\alpha,v_\alpha)$ solves \eqref{SWSystAlpha} , where 
$$\omega_\alpha^2 = \frac{1}{(qv_\alpha-1)^2}\left(m_0^2-\varepsilon_\alpha^{p-2}\right)\hskip.1cm .$$
In this case $\omega_\alpha^2 \to m_0^2$ as $\alpha \to +\infty$ and 
the construction provides a counter example to the above statement about the positivity of $u$ and $v$.

\section{A priori bounds in the critical case}\label{Apriori2}

In what follows we let $(M,g)$ be a smooth compact $4$-dimensional Riemannian manifold, 
$m_0,m_1 > 0$, and $(\omega_\alpha)_\alpha$ be a sequence in $(-m_0,m_0)$ such that 
$\omega_\alpha \to \omega$ as $\alpha \to +\infty$ for some $\omega \in [-m_0,m_0]$. Also we let 
$\bigl((u_\alpha,v_\alpha)\bigr)_\alpha$ be a sequence 
of smooth positive solutions of \eqref{SWSyst} with phases $\omega_\alpha$ and $p = 4$. Namely,
\begin{equation}\label{SWSystAlphaCrit}
\begin{cases}
\Delta_gu_\alpha + m_0^2u_\alpha = u_\alpha^3 + \omega_\alpha^2\left(qv_\alpha-1\right)^2u_\alpha\\
\Delta_gv_\alpha + \left(m_1^2+q^2u_\alpha^2\right)v_\alpha = qu_\alpha^2
\end{cases}
\end{equation}
for all $\alpha$. By the second equation in \eqref{SWSystAlphaCrit}, 
$0 \le v_\alpha \le \frac{1}{q}$ for all $\alpha$. In particular, if we let
\begin{equation}\label{Defhalpha}
h_\alpha = m_0^2 - \omega_\alpha^2\left(qv_\alpha-1\right)^2\hskip.1cm ,
\end{equation}
then $\Vert h_\alpha\Vert_{L^\infty} \le C$ for all $\alpha$, where $C > 0$ is independent of $\alpha$. 
Assume by contradiction that
\begin{equation}\label{ContrAssumptSubCptness}
\max_Mu_\alpha \to +\infty
\end{equation}
as $\alpha \to +\infty$. In what follows we let $(x_\alpha)_\alpha$ be a sequence of points in $M$, and $(\rho_\alpha)_\alpha$ be a 
sequence of positive real numbers, $0 < \rho_\alpha < i_g/7$ for all $\alpha$, where $i_g$ is the injectivity radius 
of $(M,g)$. We assume that the $x_\alpha$'s and $\rho_\alpha$'s satisfy
\begin{equation}\label{hypCdts}
\begin{cases}
\nabla u_\alpha(x_\alpha) = 0\hskip.1cm\hbox{for all}\hskip.1cm \alpha ,\\
d_g(x_\alpha,x)u_\alpha(x) \le C\hskip.1cm\hbox{for all}\hskip.1cm x\in B_{x_\alpha}(7\rho_\alpha)\hskip.1cm\hbox{and all}\hskip.1cm \alpha\hskip.1cm ,\\
\lim_{\alpha \to +\infty}\rho_\alpha\sup_{B_{x_\alpha}(6\rho_\alpha)}u_\alpha(x) = +\infty\hskip.1cm .
\end{cases}
\end{equation}
We let $\mu_\alpha$ be given by
\begin{equation}\label{DefMualpha}
\mu_\alpha = u_\alpha(x_\alpha)^{-1}\hskip.1cm .
\end{equation}
Since the $h_\alpha$'s in \eqref{Defhalpha} are $L^\infty$-bounded we can apply the asymptotic analysis in Druet and Hebey \cite{DruHeb0} and 
Druet, Hebey and V\'etois \cite{DruHebVet}. In particular, we get that $\frac{\rho_\alpha}{\mu_\alpha} \to +\infty$ as 
$\alpha \to +\infty$ and that
\begin{equation}\label{Limualpha}
\mu_\alpha u_\alpha\left(\exp_{x_\alpha}(\mu_\alpha x)\right) \to \left(1 + \frac{\vert x\vert^2}{8}\right)^{-1}
\end{equation}
in $C^1_{loc}(\mathbb{R}^4)$ as $\alpha \to +\infty$, where $\mu_\alpha$ is as in \eqref{DefMualpha}. As a consequence, 
$\mu_\alpha \to 0$ 
as $\alpha \to +\infty$. 
Now we define $\varphi_\alpha : \left(0,\rho_\alpha\right)\mapsto {\mathbb R}^+$ by 
\begin{equation}\label{eq3.1}
\varphi_\alpha(r)= \frac{1}{\left\vert \partial B_{x_\alpha}\left(r\right)\right\vert_g} 
\int_{\partial B_{x_\alpha}\left(r\right)} u_\alpha d\sigma_g\hskip.1cm ,
\end{equation}
where $\left\vert \partial B_{x_\alpha}\left(r\right)\right\vert_g$ is the volume of the sphere of center $x_\alpha$ 
and radius $r$ for the induced metric. 
Let $\Lambda = 4\sqrt{2}$. 
We define $r_\alpha\in \left[\Lambda\mu_\alpha, \rho_\alpha\right]$ by 
\begin{equation}\label{eq3.3}
r_\alpha = \sup\left\{r\in \left[\Lambda\mu_\alpha, \rho_\alpha\right]\hbox{ s.t. } \left(s\varphi_\alpha(s)\right)' \le 0 
\hskip.1cm\hbox{in}\hskip.1cm\left[\Lambda\mu_\alpha,r\right]\right\}\hskip.1cm .
\end{equation} 
It follows from \eqref{Limualpha} that 
\begin{equation}\label{eq3.4}
\frac{r_\alpha}{\mu_\alpha}\to +\infty
\end{equation}
as $\alpha \to +\infty$, while the definition of $r_\alpha$ gives that 
\begin{equation}\label{eq3.5}
r\varphi_\alpha(r) \hbox{ is non-increasing in }\left[\Lambda\mu_\alpha,r_\alpha\right]
\end{equation}
and that 
\begin{equation}\label{eq3.6}
\left(r\varphi_\alpha(r)\right)'\left(r_\alpha\right)=0\hbox{ if }r_\alpha<\rho_\alpha\hskip.1cm.
\end{equation}
Let $B_\alpha$ be defined in $M$ by
\begin{equation}\label{DefBAlpha}
B_\alpha(x) = \frac{\mu_\alpha}{\mu_\alpha^2 + \frac{d_g(x_\alpha,x)^2}{8}}\hskip.1cm ,
\end{equation}
where $\mu_\alpha$ is as in \eqref{DefMualpha}. 
The following sharp estimates, see Druet, Hebey and Robert \cite{DruHebRob} and Druet, Hebey and V\'etois \cite{DruHebVet}, hold true.

\begin{lem}\label{SharpEst} Let $(M,g)$ be a smooth compact Riemannian $4$-dimensional manifold, 
and $\bigl((u_\alpha,v_\alpha)\bigr)_\alpha$ be a sequence 
of smooth positive solutions of \eqref{SWSystAlphaCrit} such that \eqref{ContrAssumptSubCptness} holds true.
Let $(x_\alpha)_\alpha$ and $(\rho_\alpha)_\alpha$ be such that \eqref{hypCdts} hold true, and let $R \ge 6$ be such that 
$Rr_\alpha \le 6\rho_\alpha$ for all $\alpha \gg 1$. There exists $C > 0$ such that, after passing to a subsequence,
\begin{equation}\label{SharpEstCtrl}
u_\alpha(x) + d_g(x_\alpha,x)\left\vert\nabla u_\alpha(x)\right\vert \le 
C\mu_\alpha d_g(x_\alpha,x)^{-2}
\end{equation}
for all $x \in B_{x_\alpha}(\frac{R}{2}r_\alpha)\backslash\left\{x_\alpha\right\}$ and all $\alpha$, where
$\mu_\alpha$ is as in \eqref{DefMualpha}, and where $r_\alpha$ 
is as in \eqref{eq3.3}. In addition, there also exist $C > 0$ and $(\varepsilon_\alpha)_\alpha$ such that
\begin{equation}\label{SharpEstCtrlBis}
\left\vert u_\alpha - B_\alpha\right\vert \le C\mu_\alpha\left(r_\alpha^{-2}+S_\alpha\right) + \varepsilon_\alpha B_\alpha
\end{equation}
in $B_{x_\alpha}(2r_\alpha)\backslash\{x_\alpha\}$ for all $\alpha$, where $\varepsilon_\alpha \to 0$ as $\alpha \to +\infty$ 
and $S_\alpha(x) = d_g(x_\alpha,x)^{-1}$ for $x \in M\backslash\{x_\alpha\}$.
\end{lem}

Lemma \ref{SharpEst} provide a sharp control on the $u_\alpha$'s, but we need more to conclude. 
We prove that the following fundamental asymptotic estimate holds true. Lemma \ref{LemSharpAsypts} 
is the key 
estimate we need to prove the a priori bounds in the critical case discussed in this section. 

\begin{lem}\label{LemSharpAsypts} Let $(M,g)$ be a smooth compact Riemannian $4$-dimensional manifold and 
$\bigl((u_\alpha,v_\alpha)\bigr)_\alpha$ be a sequence 
of smooth positive solutions of \eqref{SWSystAlphaCrit} such that \eqref{ContrAssumptSubCptness} holds true. 
Let $(x_\alpha)_\alpha$ and $(\rho_\alpha)_\alpha$ be such that \eqref{hypCdts} holds true. Assume 
\eqref{MainAssumpt}. There holds that $r_\alpha \to 0$ 
as $\alpha \to +\infty$, where $r_\alpha$ is as in \eqref{eq3.3}. Moreover $\rho_\alpha = O\left(r_\alpha\right)$ and
\begin{equation}\label{EqtLemSharpAspt}
r_\alpha^2\mu_\alpha^{-1}u_\alpha\left(\exp_{x_\alpha}(r_\alpha x)\right) \to 
\frac{8}{\vert x\vert^2} + \mathcal{H}(x)
\end{equation}
in $C^2_{loc}\left(B_0(2)\backslash\{0\}\right)$ as $\alpha \to +\infty$, where $\mu_\alpha$ is as in \eqref{DefMualpha}, 
and $\mathcal{H}$ is a harmonic function in $B_0(2)$ which satisfies that $\mathcal{H}(0) \le 0$. 
\end{lem}

\begin{proof}[Proof of Lemma \ref{LemSharpAsypts}] Let $R \ge 6$ be such that $Rr_\alpha \le 6\rho_\alpha$ for $\alpha \gg 1$. 
We assume first that $r_\alpha\to 0$ as $\alpha\to +\infty$. For $x\in B_0(3)$ we define
\begin{eqnarray*}
\tilde{u}_\alpha(x) &=& r_\alpha^2 \mu_\alpha^{-1}u_\alpha\left(\exp_{x_\alpha}\left(r_\alpha x\right)\right)\hskip.1cm,\\
\tilde{g}_\alpha(x)&=& \left(\exp_{x_\alpha}^\star g\right)\left(r_\alpha x\right)\hskip.1cm ,\hskip.1cm\hbox{and}\\
\tilde{h}_\alpha(x)&=& h_\alpha\left(\exp_{x_\alpha}(r_\alpha x)\right)\hskip.1cm,
\end{eqnarray*}
where $h_\alpha$ is as in \eqref{Defhalpha}. 
Since $r_\alpha\to 0$ as $\alpha\to +\infty$, we have that $\tilde{g}_\alpha\to \xi$ in $C^2_{loc}(\mathbb{R}^n)$ as $\alpha\to +\infty$, where $\xi$ is the 
Euclidean metric. Thanks to 
Lemma \ref{SharpEst},
\begin{equation}\label{eq5.1}
\left\vert \tilde{u}_\alpha(x)\right\vert \le C \left\vert x\right\vert^{-2}
\end{equation}
in $B_0(\frac{R}{2})\backslash\{0\}$. By \eqref{SWSystAlphaCrit},
\begin{equation}\label{eq5.2}
\Delta_{\tilde{g}_\alpha}\tilde u_\alpha + r_\alpha^2\tilde{h}_\alpha\tilde{u}_\alpha  
= \left(\frac{\mu_\alpha}{r_\alpha}\right)^2 \tilde{u}_\alpha^3
\end{equation}
in $B_0(\frac{R}{2})$. Thanks to \eqref{eq3.4} and by standard elliptic theory, we then deduce that, after 
passing to a subsequence, 
\begin{equation}\label{eq5.3}
\tilde{u}_\alpha \to \tilde{u}
\end{equation}
in $C^2_{loc}\left(B_0(\frac{R}{2})\backslash\{0\}\right)$ as $\alpha \to +\infty$, where $\mathcal{W}$ satisfies 
$\Delta\tilde{u} = 0$ in $B_0(\frac{R}{2})\backslash\{0\}$ and $\Delta$ is the Euclidean Laplace Beltrami operator. 
Moreover, thanks to \eqref{eq5.1}, we know that 
\begin{equation}\label{eq5.5}
\left\vert\tilde{u}(x)\right\vert \le C \left\vert x\right\vert^{-2}
\end{equation}
in $B_0(\frac{R}{2})\backslash\{0\}$. Thus we can write that 
\begin{equation}\label{eq5.6}
\tilde{u}(x) = \frac{\Lambda}{\left\vert x\right\vert^2}+ \mathcal{H}(x)
\end{equation}
where $\Lambda \ge 0$ and $\mathcal{H}$ satisfies 
$\Delta\mathcal{H}=0$ in $B_0(\frac{R}{2})$. In order to see that $\Lambda = 8$, it is sufficient to integrate \eqref{eq5.2} in $B_0(1)$ to get that 
\begin{equation}\label{Eqt1LS}
-\int_{\partial B_0(1)} \partial_\nu\tilde{u}_\alpha d\sigma_{\tilde{g}_\alpha} 
= \left(\frac{\mu_\alpha}{r_\alpha}\right)^2\int_{B_0(1)}\tilde{u}_\alpha^3 dv_{\tilde{g}_\alpha}
 -r_\alpha^2 \int_{B_0(1)} \tilde{h}_\alpha\tilde{u}_\alpha dv_{\tilde{g}_\alpha}\hskip.1cm .
 \end{equation}
 By \eqref{eq5.1},
\begin{equation}\label{Eqt2LS}
\int_{B_0(1)}\tilde{u}_\alpha dv_{\tilde{g}_\alpha} \le C
\end{equation}
and by changing $x$ into $\frac{\mu_\alpha}{r_\alpha}x$, we can write that
$$\int_{B_0(1)}\tilde{u}_\alpha^3 dv_{g_\alpha} 
= r_\alpha^2\mu_\alpha^{-2}\int_{B_0(\frac{r_\alpha}{\mu_\alpha})}\hat{u}_\alpha^3 dv_{\hat{g}_\alpha}\hskip.1cm ,$$
where $\hat{u}_\alpha(x) = \mu_\alpha u_\alpha\left(\exp_{x_\alpha}(\mu_\alpha x)\right)$ and 
$\hat{g}_\alpha(x) = \left(\exp_{x_\alpha}^\star g\right)(\mu_\alpha x)$. By \eqref{Limualpha} and Lemma \ref{SharpEst}, we then get that
\begin{equation}\label{Eqt3LS}
\lim_{\alpha\to+\infty}\left(\frac{\mu_\alpha}{r_\alpha}\right)^2\int_{B_0(1)}\tilde{u}_\alpha^3 dv_{\tilde{g}_\alpha}
= 16\omega_3\hskip.1cm .
\end{equation}
Noting that by \eqref{eq5.3} and \eqref{eq5.6},
\begin{equation}\label{Eqt4LS}
\lim_{\alpha\to+\infty}\int_{\partial B_0(1)}\partial_\nu\tilde{u}_\alpha d\sigma_{\tilde{g}_\alpha} = - 2\omega_3\Lambda\hskip.1cm ,
\end{equation}
we get that $\Lambda = 8$ 
thanks to \eqref{Eqt2LS}--\eqref{Eqt4LS} by passing into the limit in \eqref{Eqt1LS} as $\alpha\to +\infty$. At this point we 
claim that there exists $\beta \in (0,1]$ and $C > 0$ such that
\begin{equation}\label{FundIneqProofAsy}
v_\alpha \le C u_\alpha^\beta\hskip.2cm\hbox{in}\hskip.1cm M
\end{equation}
for all $\alpha$. Let $x_\alpha \in M$ be a point where $\frac{v_\alpha}{u_\alpha^\beta}$ is maximum. Then, 
$$\frac{\Delta_gv_\alpha(x_\alpha)}{v_\alpha(x_\alpha)} \ge \frac{\Delta_gu_\alpha^\beta(x_\alpha)}{u_\alpha^\beta(x_\alpha)}$$
and it follows from \eqref{SWSystAlphaCrit} that 
\begin{equation}\label{BetaEqt}
\begin{split}
&q\frac{u_\alpha(x_\alpha)^2}{v_\alpha(x_\alpha)} - m_1^2 - q^2u_\alpha(x_\alpha)^2\\
&\ge -\beta(\beta-1)\frac{\vert\nabla u_\alpha(x_\alpha)\vert^2}{u_\alpha(x_\alpha)^2} + \beta u_\alpha(x_\alpha)^2 
- \beta m_0^2 + \beta \omega_\alpha^2\left(qv_\alpha(x_\alpha)-1\right)^2
\hskip.1cm .
\end{split}
\end{equation}
Choosing $\beta \in (0,1]$ such that $m_1^2-\beta m_0^2 > 0$, 
since $0 < v_\alpha \le \frac{1}{q}$, we get that $u_\alpha^\beta(x_\alpha) \ge Cv_\alpha(x_\alpha)$ for some 
$C > 0$ independent of $\alpha$. This proves \eqref{FundIneqProofAsy}. 
In what follows we let $X_\alpha$ be the $1$-form 
given by 
\begin{equation}\label{DefXAlpha}
X_\alpha(x) = \left(1 - \frac{1}{18}\hbox{Rc}_g^\sharp(x).\left(\nabla f_\alpha(x),\nabla f_\alpha(x)\right)\right)\nabla f_\alpha(x)
\hskip.1cm ,
\end{equation}
where $f_\alpha(x) = \frac{1}{2}d_g(x_\alpha,x)^2$, $\hbox{Rc}_g$ is the Ricci curvature of $g$, and $\sharp$ is the musical isomorphism. We 
apply the Pohozaev identity in Druet-Hebey \cite{DruHeb1} with the vector field $X_\alpha$ to $u_\alpha$ in $B_{x_\alpha}(r_\alpha)$. We separate the 
regular part $A_\alpha = m_0^2-\omega_\alpha^2$ from the singular part in $h_\alpha$. Then, $h_\alpha = A_\alpha + O\left(v_\alpha\right)$ and 
we get that
\begin{equation}\label{eq5.7}
\begin{split}
& \int_{B_{x_\alpha}(r_\alpha)}A_\alpha u_\alpha X_\alpha(\nabla u_\alpha)dv_g 
+ \frac{1}{8} \int_{B_{x_\alpha}(r_\alpha)} \left(\Delta_g\hbox{div}_g X_\alpha\right)u_\alpha^2dv_g \\
&+ \frac{1}{4}  \int_{B_{x_\alpha}(r_\alpha)}\left(\hbox{div}_g X_\alpha\right)A_\alpha u_\alpha^2dv_g\\
& = Q_{1,\alpha} + Q_{2,\alpha} + Q_{3,\alpha} + O\left(\int_{B_{x_\alpha}(r_\alpha)}v_\alpha u_\alpha^2dv_g\right)\\
&+ O\left( \int_{B_{x_\alpha}(r_\alpha)}v_\alpha u_\alpha\left\vert X_\alpha(\nabla u_\alpha)\right\vert dv_g\right)
\hskip.1cm ,
\end{split}
\end{equation}
where 
\begin{eqnarray*}
Q_{1,\alpha}&=& \frac{1}{4} \int_{\partial B_{x_\alpha}(r_\alpha)} \left(\hbox{div}_g X_\alpha\right)u_\alpha\partial_\nu u_\alpha d\sigma_g \\
&&- \int_{\partial B_{x_\alpha}(r_\alpha)} \left(\frac{1}{2} X_\alpha(\nu)\vert \nabla u_\alpha\vert^2 - X_\alpha(\nabla u_\alpha)\partial_\nu u_\alpha\right) d\sigma_g\hskip.1cm,
\end{eqnarray*}
$$Q_{2,\alpha} =-\sum_{i=1}^p \int_{B_{x_\alpha}(r_\alpha)} \left(\nabla X_\alpha-\frac{1}{4}\left(\hbox{div}_g X_\alpha\right) g\right)^\sharp
\left(\nabla u_\alpha,\nabla u_\alpha\right) dv_g\hskip.1cm,$$
$$Q_{3,\alpha}= \frac{1}{4}\int_{\partial B_{x_\alpha}(r_\alpha)} X_\alpha\left(\nu\right) u_\alpha^4 d\sigma_g
- \frac{1}{8} \int_{\partial B_{x_\alpha}(r_\alpha)} \left(\partial_\nu \hbox{div}_g X_\alpha\right) u_\alpha^2 d\sigma_g\hskip.1cm,$$
and $\nu$ is the unit outward normal derivative to $B_{x_\alpha}(r_\alpha)$. We have that
\begin{equation}\label{PropXAlpha}
\begin{split}
&\left\vert X_\alpha(x)\right\vert = O\left(d_g(x_\alpha,x)\right)\hskip.2cm,\hskip.2cm 
\hbox{div}_gX_\alpha(x) = n + O\left(d_g(x_\alpha,x)^2\right))\hskip.2cm,\\
&\left\vert\nabla\left(\hbox{div}_gX_\alpha\right)(x)\right\vert = O\left(d_g(x_\alpha,x)\right)\hskip.2cm ,\\
&\hbox{and}\hskip.2cm
\Delta_g\left(\hbox{div}_gX_\alpha\right)(x) = \frac{4}{3}S_g(x_\alpha) + O\left(d_g(x_\alpha,x)\right)
\hskip.1cm .
\end{split}
\end{equation}
Following Druet, Hebey and V\'etois \cite{DruHebVet} we get from Lemma \ref{SharpEst}, \eqref{eq5.7} and \eqref{PropXAlpha} that 
\begin{equation}\label{PohoCo1}
\begin{split}
Q_{1,\alpha} &= -64\omega_3\left(m_0^2-\omega^2-\frac{1}{6}S_g(x_0)\right)\mu_\alpha^2\ln\frac{r_\alpha}{\mu_\alpha}\\
& + o\left(\mu_\alpha^2\ln\frac{1}{\mu_\alpha}\right) + o\left(\mu_\alpha^2r_\alpha^{-2}\right) + O\left(\int_{B_{x_\alpha}(r_\alpha)}v_\alpha u_\alpha^2dv_g\right)\\
&+ O\left( \int_{B_{x_\alpha}(r_\alpha)}v_\alpha u_\alpha\left\vert X_\alpha(\nabla u_\alpha)\right\vert dv_g\right)\hskip.1cm ,
\end{split}
\end{equation}
where $x_\alpha \to x_0$ as $\alpha \to +\infty$. By Lemma \ref{SharpEst} and \eqref{PropXAlpha} there also holds that 
\begin{equation}\label{PohoCo2}
Q_{1,\alpha} = O\left(\mu_\alpha^2r_\alpha^{-2}\right)\hskip.1cm .
\end{equation}
At this point we decompose $v_\alpha$ into a quasi-harmonic part with nonzero Dirichlet boundary condition and a 
quasi-Poisson part with zero Dirichler boundary condition. More precisely, we write that
\begin{equation}\label{DecompVAlpha}
v_\alpha = w_{1,\alpha} + w_{2,\alpha}
\end{equation}
in $B_\alpha = B_{x_\alpha}(\hat r_\alpha)$, where $\hat r_\alpha = \frac{5}{2}r_\alpha$, and $w_{1,\alpha}$, $w_{2,\alpha}$ are given by  
\begin{equation}\label{W1AlphaEqt}
\begin{cases}
\Delta_gw_{1,\alpha} + m_1^2w_{1,\alpha} = 0\hskip.2cm\hbox{in}\hskip.1cm B_\alpha\\
w_{1,\alpha} = v_\alpha\hskip.2cm\hbox{on}\hskip.1cm \partial B_\alpha\hskip.1cm ,
\end{cases}
\end{equation}
and if $W_\alpha = \Delta_gv_\alpha + m_1^2v_\alpha$, by
\begin{equation}\label{W2AlphaEqt}
\begin{cases}
\Delta_gw_{2,\alpha} + m_1^2w_{2,\alpha} = W_\alpha\hskip.2cm\hbox{in}\hskip.1cm B_\alpha\\
w_{2,\alpha} = 0\hskip.2cm\hbox{on}\hskip.1cm \partial B_\alpha\hskip.1cm .
\end{cases}
\end{equation}
Let $G_\alpha$ be the Green's function of $\Delta_g + m_1^2$ in $B_\alpha$ with zero Dirichlet boundary condition on 
$\partial B_\alpha$. By the maximum principle, considering the Green's function on a larger ball of radius $i_g$, we 
obtain by comparison of the two Green's functions that there exists $C > 0$ such that $G_\alpha(x,y) \le Cd_g(x,y)^{-2}$ for all $x\not= y$ in $B_\alpha$. Writing that 
$$w_{2,\alpha}(x) = \int_{B_\alpha}G_\alpha(x,y)W_\alpha(y)dv_g(y)$$
it follows that
\begin{equation}\label{Est1W1Alpha}
\left\vert w_{2,\alpha}(x)\right\vert \le C \int_{B_\alpha}\frac{u_\alpha^2(y)dv_g(y)}{d_g(x,y)^2}\hskip.1cm .
\end{equation}
By \eqref{Limualpha} and Lemma \ref{SharpEst} we can write that
\begin{equation}\label{EstTotUAlpha}
u_\alpha(x) \le \frac{C\mu_\alpha}{\mu_\alpha^2 + d_g(x_\alpha,x)^2}
\end{equation}
in $B_\alpha$. Combining \eqref{Est1W1Alpha} and \eqref{EstTotUAlpha} we then get that
\begin{equation}\label{Est2W1Alpha}
\left\vert w_{2,\alpha}(x)\right\vert \le C\frac{\mu_\alpha^2\ln\left(2 + \frac{d_g(x_\alpha,x)^2}{\mu_\alpha^2}\right)}{\mu_\alpha^2+d_g(x_\alpha,x)^2}
\hskip.1cm .
\end{equation}
Independently, by the maximum principle, the $w_{1,\alpha}$'s satisfy that $0 \le w_{1,\alpha} \le \frac{1}{q}$. 
Let $\hat g_\alpha(x) = \left(\exp_{x_\alpha}^\star g\right)(\hat r_\alpha x)$ and
$\hat w_{1,\alpha}(x) = w_{1,\alpha}\left(\exp_{x_\alpha}(\hat r_\alpha x)\right)$. There holds 
\begin{equation}\label{W1AlphaEqtBis}
\begin{cases}
\Delta_{\hat g_\alpha}\hat w_{1,\alpha} + m_1^2\hat r_\alpha^2\hat w_{1,\alpha} = 0\hskip.2cm\hbox{in}\hskip.1cm B\\
w_{1,\alpha} = \hat v_\alpha\hskip.2cm\hbox{on}\hskip.1cm \partial B\hskip.1cm ,
\end{cases}
\end{equation}
where $B = B_0(1) \subset \mathbb{R}^4$, and $\hat v_\alpha(x) = v_\alpha\left(\exp_{x_\alpha}(\hat r_\alpha x)\right)$. 
At this point we claim that 
\begin{equation}\label{ConvrAlpha}
r_\alpha \to 0
\end{equation}
as $\alpha \to +\infty$. In order to prove \eqref{ConvrAlpha} we proceed by contradiction and assume that $r_\alpha \ge \delta_0 > 0$ for all 
$\alpha \gg 1$.  By 
Lemma \ref{SharpEst} and \eqref{FundIneqProofAsy}, 
\begin{equation}\label{PohoCo4}
v_\alpha \le C\mu_\alpha^\beta\hskip.2cm\hbox{in}\hskip.1cm M\backslash B_{x_\alpha}(r_\alpha)
\hskip.1cm ,
\end{equation}
where $C > 0$ is independent of $\alpha$ since we assumed $r_\alpha \ge \delta_0 > 0$. In particular, $\Vert v_\alpha\Vert_{L^\infty(\partial B_\alpha)} \to 0$ 
as $\alpha \to +\infty$. Then $\Vert\hat v_\alpha\Vert_{L^\infty(\partial B)} \to 0$ 
as $\alpha \to +\infty$, and it follows from the maximum principle and \eqref{W1AlphaEqtBis} that 
$\Vert\hat w_{1,\alpha}\Vert_{L^\infty(B)} \to 0$ 
as $\alpha \to +\infty$. In particular, $\Vert w_{1,\alpha}\Vert_{L^\infty(B_\alpha)} \to 0$ 
as $\alpha \to +\infty$. By \eqref{DecompVAlpha} and \eqref{Est2W1Alpha}, thanks to what we just obtained 
about the $w_{1,\alpha}$'s, we get that $\Vert v_\alpha\Vert_{L^\infty(B_\alpha)} \to 0$ 
as $\alpha \to +\infty$. Then, by Lemma \ref{SharpEst} and \eqref{PropXAlpha} we get that
\begin{equation}\label{PohoCo10}
\begin{split}
&\int_{B_{x_\alpha}(r_\alpha)}v_\alpha u_\alpha^2dv_g = o\left(\mu_\alpha^2\ln\frac{1}{\mu_\alpha}\right)\\
&\int_{B_{x_\alpha}(r_\alpha)} u_\alpha v_\alpha \vert X_\alpha(\nabla u_\alpha)\vert dv_g = o\left(\mu_\alpha^2\ln\frac{1}{\mu_\alpha}\right)
\hskip.1cm .
\end{split}
\end{equation}
and by \eqref{PohoCo1} and \eqref{PohoCo2}, we obtain a contradiction with \eqref{MainAssumpt}. This proves \eqref{ConvrAlpha}.
By \eqref{eq5.3}, \eqref{eq5.5} and \eqref{eq5.6} we get with \eqref{ConvrAlpha} that
\begin{equation}\label{ComputQ1Alpha}
Q_{1,\alpha} = -\left(128\omega_3\mathcal{H}(0) + o(1)\right)\mu_\alpha^2r_\alpha^{-2}
\hskip.1cm .
\end{equation}
Now we distinguish the two cases:

\medskip (i) $r_\alpha^2\ln\frac{r_\alpha}{\mu_\alpha} \to 0$ as $\alpha\to+\infty$, and

\medskip (ii) $r_\alpha^2\ln\frac{r_\alpha}{\mu_\alpha} \ge \delta_0 > 0$ for all $\alpha$.

\medskip\noindent In case (i), since $v_\alpha = O(1)$, we get from Lemma \ref{SharpEst} and \eqref{PropXAlpha} that
\begin{equation}\label{ConcludEqtLemFond1}
\begin{split}
&\int_{B_{x_\alpha}(r_\alpha)}v_\alpha u_\alpha^2dv_g = O\left(\mu_\alpha^2\ln\frac{r_\alpha}{\mu_\alpha}\right)\hskip.1cm ,\\
&\int_{B_{x_\alpha}(r_\alpha)}v_\alpha u_\alpha\vert X_\alpha(\nabla u_\alpha)\vert dv_g = O\left(\mu_\alpha^2\ln\frac{r_\alpha}{\mu_\alpha}\right)
\hskip.1cm .
\end{split}
\end{equation}
Since there also holds that $r_\alpha^2\ln\frac{1}{\mu_\alpha} \to +\infty$ it follows from \eqref{PohoCo1}, \eqref{ComputQ1Alpha} and \eqref{ConcludEqtLemFond1} 
that $\mathcal{H}(0) = 0$. Now we assume (ii). From (ii) we get that $r_\alpha \ge C(\ln\frac{1}{\mu_\alpha})^{-1/2}$ and by \eqref{FundIneqProofAsy} we obtain that 
$$v_\alpha \le C\left(\mu_\alpha\ln\frac{1}{\mu_\alpha}\right)^\beta
\hskip.2cm\hbox{in}\hskip.1cm M\backslash B_{x_\alpha}(r_\alpha)\hskip.1cm .$$
In particular, $\Vert v_\alpha\Vert_{L^\infty(\partial B_\alpha)} \to 0$ 
as $\alpha \to +\infty$. Then $\Vert\hat v_\alpha\Vert_{L^\infty(\partial B)} \to 0$ 
as $\alpha \to +\infty$, and it follows from the maximum principle and \eqref{W1AlphaEqtBis} that 
$\Vert\hat w_{1,\alpha}\Vert_{L^\infty(B)} \to 0$ 
as $\alpha \to +\infty$. In particular, $\Vert w_{1,\alpha}\Vert_{L^\infty(B_\alpha)} \to 0$ 
as $\alpha \to +\infty$ and we get with \eqref{DecompVAlpha}, Lemma \ref{SharpEst}, and \eqref{PropXAlpha}, that 
\begin{equation}\label{ConcludEqtBisLemFond1}
\begin{split}
&\int_{B_{x_\alpha}(r_\alpha)}v_\alpha u_\alpha^2dv_g = \int_{B_{x_\alpha}(r_\alpha)}w_{2,\alpha}u_\alpha^2dv_g 
+ o\left(\mu_\alpha^2\ln\frac{r_\alpha}{\mu_\alpha}\right)\hskip.1cm ,\\
&\int_{B_{x_\alpha}(r_\alpha)}v_\alpha u_\alpha\vert X_\alpha(\nabla u_\alpha)\vert dv_g =\\
&\hskip.4cm \int_{B_{x_\alpha}(r_\alpha)}w_{2,\alpha}u_\alpha\vert X_\alpha(\nabla u_\alpha)\vert dv_g + 
o\left(\mu_\alpha^2\ln\frac{r_\alpha}{\mu_\alpha}\right)
\hskip.1cm .
\end{split}
\end{equation}
There holds,
\begin{equation}\label{PDEEqtW2Alpha}
\Delta_gw_{2,\alpha} + m_1^2w_{2,\alpha} = q\left(1-qv_\alpha\right)u_\alpha^2
\hskip.1cm .
\end{equation}
Let $\eta: \mathbb{R}^n \to \mathbb{R}$ be such that $\eta$ is smooth, $0 \le \eta \le 1$, $\eta = 1$ in $B_0(1)$, and 
$\eta = 0$ in $\mathbb{R}^n\backslash B_0(2)$. We define 
\begin{equation}\label{PohoCo3}
\eta_\alpha(x) = \eta\left(\frac{d_g(x_\alpha,x)}{r_\alpha}\right)
\end{equation}
so that $\eta_\alpha = 1$ in $B_{x_\alpha}(r_\alpha)$ and $\eta_\alpha = 0$ in $M\backslash B_{x_\alpha}(2r_\alpha)$. 
By H\"older's inequalities,
\begin{equation}\label{PohoCo5}
\begin{split}
&\int_{B_{x_\alpha}(r_\alpha)}w_{2,\alpha}u_\alpha^2dv_g \le \left(\int_{B_{x_\alpha}(r_\alpha)}w_{2,\alpha}^4dv_g\right)^{1/4}
\left(\int_{B_{x_\alpha}(r_\alpha)}u_\alpha^{8/3}dv_g\right)^{3/4}\hskip.2cm\hbox{and}\\
&\int_{B_{x_\alpha}(r_\alpha)} u_\alpha w_{2,\alpha}\vert X_\alpha(\nabla u_\alpha)\vert dv_g\\
&\le \left(\int_{B_{x_\alpha}(r_\alpha)}w_{2,\alpha}^4dv_g\right)^{1/4}
\left(\int_{B_{x_\alpha}(r_\alpha)}\vert u_\alpha X_\alpha(\nabla u_\alpha)\vert^{4/3}dv_g\right)^{3/4}\hskip.1cm ,
\end{split}
\end{equation}
while by Lemma \ref{SharpEst} and \eqref{PropXAlpha} there holds that
\begin{equation}\label{PohoCo6}
\begin{split}
&\int_{B_{x_\alpha}(r_\alpha)}u_\alpha^{8/3}dv_g = O\left(\mu_\alpha^{4/3}\right)\hskip.1cm \hbox{and} \\
&\int_{B_{x_\alpha}(r_\alpha)}\vert u_\alpha X_\alpha(\nabla u_\alpha)\vert^{4/3}dv_g = O\left(\mu_\alpha^{4/3}\right)
\hskip.1cm .
\end{split}
\end{equation}
Multiplying \eqref{PDEEqtW2Alpha} by $\eta_\alpha^2w_{2,\alpha}$, and integrating over $M$, we get that 
\begin{equation}\label{PohoCo7}
\int_M\left(\Delta_gw_{2,\alpha}+m_1^2w_{2,\alpha}\right)\eta_\alpha^2w_{2,\alpha}dv_g \le q\int_Mu_\alpha^2w_{2,\alpha}\eta_\alpha^2dv_g
\hskip.1cm .
\end{equation}
By H\"older's and Sobolev inequalities, and by \eqref{PohoCo6}, 
\begin{equation}\label{PohoCo8}
\int_Mu_\alpha^2w_{2,\alpha}\eta_\alpha^2dv_g \le C\mu_\alpha\Vert\eta_\alpha w_{2,\alpha}\Vert_{H^1}
\end{equation}
and it follows from \eqref{PohoCo7} and \eqref{PohoCo8} that
\begin{equation}\label{PohoCo9} 
\Vert\eta_\alpha w_{2,\alpha}\Vert_{H^1}^2 \le \int_M\vert\nabla\eta_\alpha\vert^2w_{2,\alpha}^2dv_g + C\mu_\alpha\Vert\eta_\alpha w_{2,\alpha}\Vert_{H^1}
\hskip.1cm .
\end{equation}
By \eqref{Est2W1Alpha}, since $\vert\nabla\eta_\alpha\vert \le Cr_\alpha^{-1}$, we get that
$$\int_M\vert\nabla\eta_\alpha\vert^2w_{2,\alpha}^2dv_g \le Cr_\alpha^2\left(\frac{\mu_\alpha^2}{r_\alpha^2}\ln\left(\frac{r_\alpha}{\mu_\alpha}\right)\right)^2$$
and by \eqref{eq3.4} it follows that
$$\int_M\vert\nabla\eta_\alpha\vert^2w_{2,\alpha}^2dv_g = o\left(\mu_\alpha^2\ln^2\frac{1}{\mu_\alpha}\right)\hskip.1cm .$$
Coming back to \eqref{PohoCo9}, it follows that
\begin{equation}\label{NormEstimate}
 \Vert\eta_\alpha w_{2,\alpha}\Vert_{H^1} = o\left(\mu_\alpha\ln\frac{1}{\mu_\alpha}\right)\hskip.1cm .
 \end{equation}
 By \eqref{ConcludEqtBisLemFond1}, \eqref{PohoCo5} and \eqref{PohoCo6}, we then get with \eqref{NormEstimate} that 
 \begin{equation}\label{RemainderEstimates}
 \begin{split}
 &\int_{B_{x_\alpha}(r_\alpha)} v_\alpha u_\alpha^2dv_g = o\left(\mu_\alpha^2\ln\frac{1}{\mu_\alpha}\right)\hskip.2cm\hbox{and}\\
 &\int_{B_{x_\alpha}(r_\alpha)}u_\alpha v_\alpha \vert X_\alpha(\nabla u_\alpha)\vert dv_g  = o\left(\mu_\alpha^2\ln\frac{1}{\mu_\alpha}\right)
 \hskip.1cm .
 \end{split}
 \end{equation}
Coming back to \eqref{PohoCo1} and \eqref{ComputQ1Alpha} it follows that 
 \begin{equation}\label{FinalH0}
 \mathcal{H}(0) = \frac{1}{64}\left(m_0^2-\omega^2-\frac{1}{6}S_g(x_0)\right)\lim_{\alpha\to+\infty}r_\alpha^2\ln\frac{r_\alpha}{\mu_\alpha}
 \hskip.1cm .
 \end{equation}
 By \eqref{FinalH0} we get that $\mathcal{H}(0) \le 0$. At this point it remains to prove that $\rho_\alpha = O\left( r_\alpha\right)$. We prove that 
 $\rho_\alpha = r_\alpha$. If not the case, then $r_\alpha < \rho_\alpha$ and we get with \eqref{eq3.6} that $\left(r\varphi(r)\right)^\prime(1) = 0$, where 
 \begin{eqnarray*} \varphi(r) 
 &=& \frac{1}{\omega_3r^3} \int_{\partial B_0(r)}\tilde u d\sigma\\
 &=& \frac{8}{r^2} + \mathcal{H}(0)\hskip.1cm .
\end{eqnarray*}
Hence $\mathcal{H}(0) = 8$ and we get a contradiction with $\mathcal{H}(0) \le 0$. In other words, $\rho_\alpha  = r_\alpha$ for all $\alpha \gg 1$. 
This ends the proof of the lemma.
\end{proof}

Thanks to Lemma \ref{LemSharpAsypts} we can now prove the uniform bounds in 
Theorem \ref{Thm2}. This is the subject of what follows. 

\begin{proof}[Proof of the uniform bounds in Theorem \ref{Thm2}] 
Let $(M,g)$ be a smooth compact Riemannian $4$-dimensional manifold and 
$\bigl((u_\alpha,v_\alpha)\bigr)_\alpha$ be a sequence 
of smooth positive solutions of \eqref{SWSystAlphaCrit} such that 
\eqref{MainAssumpt} holds true. By Druet, Hebey and V\'etois \cite{DruHebVet} 
there exists $C > 0$ such that for any $\alpha$ the following holds true: 
there exist $N_\alpha\in \mathbb{N}^\star$ and $N_\alpha$ critical points of $u_\alpha$, 
denoted by $\left(x_{1,\alpha}, x_{2,\alpha}, \dots, x_{N_\alpha,\alpha}\right)$, such that 
\begin{equation}\label{Eqt1Pr}
d_g\left(x_{i,\alpha},x_{j,\alpha}\right)u_\alpha(x_{i,\alpha}) \ge 1
\end{equation}
for all $i, j \in \left\{1,\dots,N_\alpha\right\}$, $i\neq j$, and 
\begin{equation}\label{Eqt2Pr}
\left(\min_{i=1,\dots,N_\alpha} d_g\left(x_{i,\alpha}, x\right)\right)u_\alpha(x) \le C
\end{equation}
for all $x\in M$ and all $\alpha$. We define
\begin{equation}\label{eqconcl1}
d_\alpha = \min_{1 \le i <   j \le N_\alpha} d_g\left(x_{i,\alpha},x_{j,\alpha}\right)\hskip.1cm.
\end{equation}
If $N_\alpha=1$, we set $d_\alpha= \frac{1}{4}i_g$, where $i_g$ is the injectivity radius of $(M,g)$. We claim that 
\begin{equation}\label{eqconcl2}
d_\alpha\not\to 0
\end{equation}
as $\alpha \to +\infty$. In order to prove this claim, we proceed by contradiction. Assuming on the contrary that 
$d_\alpha\to 0$ as $\alpha\to +\infty$, we see that $N_\alpha\ge 2$ for $\alpha$ large, and we can 
thus assume that the concentration points are ordered in such a way that 
\begin{equation}\label{eqconcl3}
d_\alpha = d_g\left(x_{1,\alpha},x_{2,\alpha}\right) \le d_g\left(x_{1,\alpha},x_{3,\alpha}\right) \le \dots 
\le d_g\left(x_{1,\alpha},x_{N_\alpha,\alpha}\right)\hskip.1cm.
\end{equation}
We set, for $x\in B_0(\delta d_\alpha^{-1})$, $0<\delta<\frac{1}{2}i_g$ fixed, 
\begin{eqnarray*} 
\hat{u}_\alpha(x) &=& d_\alpha u_\alpha \left(\exp_{x_{1,\alpha}}(d_\alpha x)\right)\hskip.1cm,\\
\hat{h}_\alpha (x)&=& h_\alpha\left(\exp_{x_{1,\alpha}}(d_\alpha x)\right)\hskip.1cm ,\hskip.1cm\hbox{and}\\
\hat{g}_\alpha(x)&=& \left(\exp_{x_{1,\alpha}}^\star g\right)(d_\alpha x)\hskip.1cm.
\end{eqnarray*}
It is clear that $\hat{g}_\alpha\to \xi$ in $C^2_{loc}(\mathbb{R}^n)$ as $\alpha\to +\infty$ since 
$d_\alpha\to 0$ as $\alpha\to +\infty$. Thanks to \eqref{SWSystAlphaCrit} we have that 
\begin{equation}\label{eqconcl4}
\Delta_{\hat{g}_\alpha}\hat{u}_\alpha 
+ d_\alpha^2\hat{h}_\alpha\hat{u}_\alpha = \hat{u}_\alpha^3
\end{equation}
in $B_0(\delta d_\alpha^{-1})$, for all $i$. 
For any $R>0$, we also let $1\le N_{R,\alpha}\le N_\alpha$ be such that 
\begin{eqnarray*} 
&&d_g(x_{1,\alpha},x_{i,\alpha}) \le Rd_\alpha\hskip.1cm\hbox{for}\hskip.1cm 1\le i\le N_{R,\alpha}\hskip.1cm ,\hskip.1cm\hbox {and}\\
&&d_g\left(x_{1,\alpha},x_{i,\alpha}\right) > Rd_\alpha \hbox{ for }N_{R,\alpha}+1\le i\le N_{\alpha}\hskip.1cm.
\end{eqnarray*}
Such a $N_{R,\alpha}$ does exist thanks to \eqref{eqconcl3}. We also have that $N_{R,\alpha}\ge 2$ for all $R>1$ and 
that $(N_{R,\alpha})_\alpha$ is uniformly bounded for all $R>0$ thanks to \eqref{eqconcl1}. 
In the sequel, we set 
$$\hat{x}_{i,\alpha} = d_\alpha^{-1} \exp_{x_{1,\alpha}}^{-1}(x_{i,\alpha})$$
for all $1\le i\le N_\alpha$ such that $d_g(x_{1,\alpha}, x_{i,\alpha})\le \frac{1}{2}i_g$. Thanks to \eqref{Eqt2Pr}, for any $R>1$, there exists $C_R>0$ such that 
\begin{equation}\label{eqconcl5}
\sup_{B_0(R)\setminus \bigcup_{i=1}^{N_{2R,\alpha}} B_{\hat{x}_{i,\alpha}}\left(\frac{1}{R}\right)}\hat{u}_\alpha \le C_R\hskip.1cm.
\end{equation}
By the Harnack inequality in Druet, Hebey and V\'etois \cite{DruHebVet}, 
for any $R>1$, there exists $D_R>1$ such that 
\begin{equation}\label{eqconcl6}
\left\Vert \nabla \hat{u}_\alpha\right\Vert_{L^\infty\left(\Omega_{R,\alpha}\right)} 
\le D_R \sup_{\Omega_{R,\alpha}} \hat{u}_\alpha 
\le D_R^2 \inf_{\Omega_{R,\alpha}} \hat{u}_\alpha\hskip.1cm ,
\end{equation}
where
$$\Omega_{R,\alpha}= B_0(R)\setminus \bigcup_{i=1}^{N_{2R,\alpha}} B_{\hat{x}_{i,\alpha}}\left(\frac{1}{R}\right)\hskip.1cm .$$
Assume first that, for some $R>0$, there exists $1\le i\le N_{R,\alpha}$ such that 
\begin{equation}\label{AddedEqtAss1}
\hat{u}_\alpha(\hat{x}_{i,\alpha}) = O(1)\hskip.1cm.
\end{equation}
The two first equations in \eqref{hypCdts} are satisfied by the sequences $x_\alpha = x_{i,\alpha}$ and $\rho_\alpha= \frac{1}{8}d_\alpha$. Then it follows from 
\eqref{Limualpha} that the last equation in \eqref{hypCdts} cannot hold and 
thus that $( \hat{u}_\alpha)_\alpha$ is uniformly bounded in $B_{\hat{x}_{i,\alpha}}(\frac{3}{4})$. 
In particular, by standard elliptic theory, and thanks to \eqref{eqconcl4}, 
$(\hat{u}_\alpha)_\alpha$ is uniformly bounded in $C^1\left(B_{\hat{x}_{i,\alpha}}(\frac{1}{2})\right)$. 
Since, by \eqref{Eqt1Pr}, we have that 
$$\vert \hat{x}_{i,\alpha}\vert^{\frac{n-2}{2}}\vert \hat {u}_\alpha(\hat{x}_{i,\alpha})\vert \ge 1\hskip.1cm,$$
we get the existence of some $\delta_i>0$ such that 
$$\vert \hat{u}_\alpha\vert \ge \frac{1}{2}\vert \hat{x}_{i,\alpha}\vert^{1-\frac{n}{2}} \ge \frac{1}{2} R^{1-\frac{n}{2}}$$
in $B_{\hat{x}_{i,\alpha}}(\delta_i)$. Assume now that, for some $R>0$, there exists $1\le i\le N_{R,\alpha}$ such that 
\begin{equation}\label{AddedEqtAss2}
\vert \hat{u}_\alpha(\hat{x}_{i,\alpha})\vert \to +\infty
\end{equation}
as $\alpha \to +\infty$. Since \eqref{hypCdts} is satisfied by the sequences 
$x_\alpha = x_{i,\alpha}$ and $\rho_\alpha= \frac{1}{8}d_\alpha$, it follows from Lemma \ref{LemSharpAsypts} that the sequence 
$(\vert \hat{u}_\alpha(\hat{x}_{i,\alpha})\vert\times\vert\hat{u}_\alpha\vert)_\alpha$ is uniformly bounded in 
$$\hat{\Omega}_\alpha = B_{\hat{x}_{i,\alpha}}(\tilde{\delta}_i)\backslash B_{\hat{x}_{i,\alpha}}(\frac{\tilde{\delta}_i}{2})$$ 
for some $\tilde{\delta}_i > 0$. 
Thus, using (\ref{eqconcl6}), we can deduce that these two situations are mutually exclusive 
in the sense that either \eqref{AddedEqtAss1} holds true for all $i$ or \eqref{AddedEqtAss2} holds true for all $i$. 
Now we split the conclusion of the 
proof  into two cases. 

\medskip In the first case we assume that there exist $R>0$ and $1\le i \le N_{R,\alpha}$ such that 
$\hat{u}_\alpha(\hat{x}_{i,\alpha}) = O(1)$. Then, thanks to the above discussion, we get that 
$\hat{u}_\alpha(\hat{x}_{j,\alpha}) = O(1)$ 
for all $1\le j \le N_{R,\alpha}$ and all $R>0$. As above, we get that $(\hat{u}_\alpha)_\alpha$ 
is uniformly bounded in $C^1_{loc}(\mathbb{R}^4)$. Thus, by standard elliptic theory, there exists a subsequence of 
$(\hat{u}_\alpha)_\alpha$ which converges in $C^1_{loc}(\mathbb{R}^4)$ to some $\hat{u}$ solution of 
$\Delta \hat{u} = \hat{u}^3$ 
in $\mathbb{R}^4$. By the above discussion, 
$\vert{u}\vert$ possesses at least two critical points, namely $0$ and $\hat{x}_2$, the limit of $\hat{x}_{2,\alpha}$. This is absurd thanks to the 
classification of Caffarelli, Gidas and Spruck \cite{CafGidSpr}. 

\medskip In the second case we assume that there exist $R > 0$ and $1\le i \le N_{R,\alpha}$ such that 
$\vert \hat{u}_\alpha(\hat{x}_{i,\alpha})\vert \to +\infty$ as $\alpha \to +\infty$. Then, thanks to the above discussion,
$\hat{u}_\alpha(\hat{x}_{j,\alpha}) \to +\infty$ 
as $\alpha \to +\infty$, for all $1\le j \le N_{R,\alpha}$ and all $R>0$. By \eqref{eqconcl4} we have that
$$\Delta_{\hat{g}_\alpha}\hat{v}_\alpha 
+ d_\alpha^2\hat{h}_\alpha\hat{v}_\alpha = 
\frac{1}{\hat{u}_\alpha(0)^2}\hat{v}_\alpha^3\hskip.1cm ,$$
where $\hat{v}_\alpha = \hat{u}_\alpha(0)\hat{u}_\alpha$. Applying Lemma \ref{LemSharpAsypts} and 
standard elliptic theory, and thanks to \eqref{eqconcl6} and to 
the above discussion, one easily checks that, after passing to a subsequence, 
$\hat{u}_\alpha(0) \hat{u}_\alpha \to \hat{G}$ 
in $C^1_{loc}\left(\mathbb{R}^n\backslash\{\hat{x}_i\}_{i\in I}\right)$ as $\alpha\to +\infty$, where
$I= \left\{1,\dots, \lim_{R\to +\infty}\lim_{\alpha\to +\infty} N_{R,\alpha}\right\}$
and, for any $R>0$, 
$$\hat{G}(x) = \sum_{i=1}^{\tilde{N}_R} \frac{\Lambda_i}{\vert x-\hat{x}_i\vert^2} + \hat{H}_R(x)$$
in $B_0(R)$, where $2 \le \tilde{N}_R \le N_{2R}$ is such that $\vert \hat{x}_{\tilde{N}_R}\vert \le R$ and $\vert \hat{x}_{\tilde{N}_R +1}\vert > R$, 
where $N_{2R,\alpha} \to N_{2R}$ as $\alpha\to+\infty$, where $\lambda_i > 0$, and where 
$\hat{H}_R$ is a harmonic function in $B_0(R)$. 
Since $\hat{G}\ge 0$, we can write thanks to the maximum principle that, in a neighbourhood of the origin, 
$$\hat{G}(x) =\frac{\Lambda_1}{\vert x\vert^{n-2}} + \hat{H}(x)
\hskip.1cm ,$$
where $\hat{H}(0) \ge \Lambda_2- \Lambda_1R^{-2} - \Lambda_2(R-1)^{-2}$. 
Choosing $R$ large enough, we can ensure that 
$\hat{H}(0)>0$ and this is in contradiction with Lemma  \ref{LemSharpAsypts}.

\medskip By the above discussion we get that \eqref{eqconcl2} holds true. Clearly, this implies that 
$(N_\alpha)_\alpha$ is uniformly bounded. Let $(x_\alpha)_\alpha$ be a sequence of maximal points 
of $u_\alpha$. Thanks to \eqref{ContrAssumptSubCptness} and to \eqref{eqconcl2}, we clearly have that 
\eqref{hypCdts} holds true for the sequences $(x_\alpha)_\alpha$ and 
$\rho_\alpha=\delta$ for some $\delta>0$ fixed. This clearly contradicts Lemma 
\ref{LemSharpAsypts} and thus concludes the proof of the uniform bounds in Theorem \ref{Thm2}. 
\end{proof}

Existence and nonexistence of a priori estimates for 
critical elliptic Schr\"odinger type equations on manifolds 
have been investigated by Berti-Malchiodi \cite{BerMal}, 
Brendle \cite{Bre,BreSur}, 
Brendle and Marques \cite{BreMar}, 
Br\'ezis and Li \cite{BreLi}, 
Druet \cite{Dru1,Dru2}, 
Druet and Hebey \cite{DruHebIMRS, DruHeb0, DruHeb1}, 
Druet, Hebey, and V\'etois \cite{DruHebVet}, 
Druet and Laurain \cite{DruLau}, 
Hebey \cite{Heb1,Heb2}, 
Khuri, Marques and Schoen \cite{KhuMarSch}, 
Li and Zhang \cite{LiZha,LiZha2},
Li and Zhu \cite{LiZhu}, 
Marques \cite{Mar}, 
Micheletti, Pistoia and V\'etois \cite{MicPisVet}, 
Schoen \cite{Sch3,Sch4}, and 
V\'etois \cite{Vet}. In the subcritical case, a priori estimates for Schr\"odinger equations go back   
to the seminal work by Gidas and Spruck \cite{GidSpr}. The above list is not exhaustive.

\medskip\noindent{\bf Acknowledgments:} The first author was partially supported by the ANR grant 
ANR-08-BLAN-0335-01. The authors are indebted to Olivier Druet for a very valuable 
idea which has helped to improve the manuscript. They are also indebted to Valdimir Georgescu and Fr\'ed\'eric Robert 
for several 
interesting discussions on this work.

\end{document}